\documentclass{amsproc}
\usepackage{amssymb}
\usepackage{amsbsy}
\usepackage{amscd}

\usepackage{nicefrac}

\makeatletter
\renewcommand{\thesubsection}{\thesection(\@roman\c@subsection)}
\makeatother
\usepackage{verbatim}
\usepackage{version}
\usepackage{color}
\newenvironment{NB}{
\color{red}{\bf NB}. \footnotesize
}{}
\excludeversion{NB}
\newenvironment{NB2}{
\color{blue}{\bf NB}. \footnotesize
}{}
\newtheorem{Theorem}[equation]{Theorem}
\newtheorem{Corollary}[equation]{Corollary}
\newtheorem{Lemma}[equation]{Lemma}
\newtheorem{Proposition}[equation]{Proposition}

\theoremstyle{definition}

\theoremstyle{remark}
\newtheorem{Remark}[equation]{Remark}
\newtheorem{Remarks}[equation]{Remarks}

\numberwithin{equation}{section}

\newcommand{\thmref}[1]{Theorem~\ref{#1}}
\newcommand{\secref}[1]{\S\ref{#1}}
\newcommand{\lemref}[1]{Lemma~\ref{#1}}
\newcommand{\propref}[1]{Proposition~\ref{#1}}
\newcommand{\corref}[1]{Corollary~\ref{#1}}
\newcommand{\subsecref}[1]{\S\ref{#1}}

\newcommand{\remsref}[1]{Remarks~\ref{#1}}
\newcommand{\defeq}{\overset{\operatorname{\scriptstyle def.}}{=}}
\newcommand{\C}{{\mathbb C}}
\newcommand{\Z}{{\mathbb Z}}
\newcommand{\Q}{{\mathbb Q}}
\newcommand{\R}{{\mathbb R}}
\newcommand{\proj}{{\mathbb P}}

\newcommand{\SL}{\operatorname{\rm SL}}
\newcommand{\SU}{\operatorname{\rm SU}}

\newcommand{\Hom}{\operatorname{Hom}}

\newcommand{\Wedge}{{\textstyle \bigwedge}}
\newcommand{\ch}{\operatorname{ch}}
\newcommand{\shfO}{\mathcal O}
\newcommand{\Snorm}{S^{\mathrm{norm}}}
\newcommand{\Tnorm}{T^{\mathrm{norm}}}

\newcommand{\calE}{\mathcal E}
\newcommand{\bt}{\mathbf t}
\newcommand{\ba}{\mathbf a}
\newcommand{\bq}{\mathbf q}
\newcommand{\uE}{\calE/[0]}
\newcommand{\cP}{\mathcal P}

\renewcommand{\MR}[1]{}

\usepackage[bookmarks=false]{hyperref}

\usepackage{url}

\begin{document}

\author{Hiraku Nakajima}
\title[Refined Chern-Simons theory and Hilbert schemes]
{Refined Chern-Simons theory and Hilbert schemes of points on the plane
}
\address{Research Institute for Mathematical Sciences,
Kyoto University, Kyoto 606-8502,
Japan}
\email{nakajima@kurims.kyoto-u.ac.jp}
\thanks{Supported by the Grant-in-aid
for Scientific Research (No.23340005), JSPS, Japan.
}
\subjclass[2010]{Primary 14C05, 57M27; Secondary 14D21, 33D52}
\keywords{Hilbert schemes of points, Chern-Simons theory}
\begin{abstract}
  Aganagic and Shakirov propose a refinement of the $SU(N)$ Chern-Simons
  theory for links in three manifolds with $S^1$-symmetry, such as torus
  knots in $S^3$, based on deformation of the $S$ and $T$ matrices,
  originally found by Kirillov and Cherednik. We relate the large $N$
  limit of the $S$ matrix to the Hilbert schemes of points on the affine
  plane. As an application, we find an explicit formula for the Euler
  characteristics of the universal sheaf, applied arbitrary Schur
  functor.
\end{abstract}
\dedicatory{Dedicated to Professor Igor Frenkel on the occasion of his
  sixtieth birthday}
\maketitle

\section*{Introduction}

Witten realized polynomial invariants of knots and links, such as the
Jones polynomial by the Chern-Simons theory on $S^3$ \cite{Witten}.
This construction also gave invariants for three-manifolds and links
in them.
Recently Aganagic and Shakirov propose a refinement of the $\SU(N)$
Chern-Simons theory \cite{AS,AS2} for links in three manifolds with
$S^1$-symmetry, such as torus knots in $S^3$. 
As their refinement is based on a deformation of the $S$ and $T$
matrices, let us recall them briefly.
The $S$ and $T$ matrices give the action of $\SL(2,\Z)$ on the quantum
Hilbert space of the theory on $T^2$. In the original Chern-Simons
theory, the Hilbert space is naturally identified with the space of
level $k$ integrable highest weight representations of the affine Lie
algebra corresponding to $\SU(N)$, and $S$ and $T$ come from modular
transformations of characters.
For example, the invariant of the $(n,m)$ torus knot, colored with a
representation corresponding to a partition $\lambda$, is given by
\begin{equation*}
  \langle \emptyset | K (N_{\lambda\bullet}^\bullet) K^{-1} S | \emptyset\rangle.
\end{equation*}
Here $K$ is a matrix corresponding to an element in $\SL(2,\Z)$
mapping $(0,1)$ to $(n,m)$, which is a product in $S$, $T$. And
$N_{\lambda\bullet}^\bullet$ is the matrix given by the fusion product
of the representation for $\lambda$. The empty partition $\emptyset$
corresponds to the vacuum representation at level $k$.
The entry $S_{\lambda\mu}$ of the $S$ matrix is the invariant for the
Hopf link colored with $\lambda$ and $\mu$.

In the refined theory, the $S$, $T$ and $N_{\lambda\bullet}^\bullet$
are replaced by matrices found by Kirillov \cite{Kir} and Cherednik
\cite{Che} in connection with Macdonald polynomials. The Hilbert space
has the same size, but is identified more naturally with the space of
Macdonald polynomials of type $SU(N)$ with parameters $q$, $t$
satisfying the relation $q^k t^N = 1$.

Let us note that Cherednik gives a double affine Hecke algebra
approach to the refined Chern-Simons theory \cite{CherJ}. His approach
is connected to Hilbert schemes of points by Gorsky-Negut
\cite{2013arXiv1304.3328G} after this article was posted to the arXiv.

The Chern-Simons theory, both original and refined, depends on two
positive integers rank $N$ and level $k$. But it is expected to give a
Laurent polynomial in $\sqrt{q}$, $\sqrt{t}$ and $t^{N/2}$ for a link
in $S^3$. This three variable polynomial is specialized to the colored
HOMFLY polynomial at $q=t$.
When we restrict ourselves to consider only torus knots colored with
the vector representation, the three variable polynomial is
conjectured to be the same as the Poincar\'e polynomial of the link
homology group defined by Khovanov and Rozansky \cite{KR1,KR2} (see
also \cite{Kh}).

In this paper we relate the large $N$ limit of the $S$ matrix to the
Hilbert schemes of points on the affine plane $\C^2$. More precisely,
we identify it with an equivariant Euler characteristics of tensor
products of two copies of universal sheaves over Hilbert schemes,
after Schur functors are applied. (See \thmref{thm:main}.)

This relation is natural from various points of view.
At first by the work of Haiman \cite{Haiman} the Macdonald polynomials
correspond to elements in the equivariant $K$-group of the Hilbert
schemes with respect to the $\C^*\times \C^*$-action given by
$\C^*\times\C^*$ fixed points. The integrality and positivity, which
one should expect from the conjectural relation to link homology
groups, are explained best in this context.

Secondly there is a conjecture on the homology group of the link of a
plane curve singularity in terms of the virtual Poincar\'e polynomials
of the Hilbert schemes of points on the singularity by Oblomkov,
Rasmussen and Shende \cite{ORS} (see also \cite{GORS}).
For torus knots this theory is related to the Hilbert schemes of
points on $\C^2$ via the Cherednik algebra
\cite[Conjectures~7,8]{ORS}.

Thirdly the equivariant $K$-group of the Hilbert scheme is naturally
identified with equivariant $K$-group of $\C^{2n}$ with respect to the
$S_n\times\C^*\times\C^*$-action. This is purely combinatorial
setting, and resembles to the setting in \cite{Kh}.

Finally our Hilbert scheme Euler characteristic is the $K$-theoretic
$U(1)$-Nekrasov's partition function for the $4$-dimensional gauge
theory on $\R^4$. It is related to the Chern-Simons theory on $S^3$ via
physical equivalences between the following four theories:
\begin{enumerate}
\item Nekrasov's partition function for the $U(1)$-gauge theory on $\R^4$,
\item the topological string theory on the resolved conifold, the
  total space of $\shfO(-1)\oplus \shfO(-1)\to \proj^1$,
\item the topological string theory on $T^*S^3$,
\item the Chern-Simons theory on $S^3$.
\end{enumerate}
The relation between (1) and (2) is given by the so-called geometric
engineering \cite{Gengineer}.
The topological string theory for the resolved conifold (2) and one for
$T^*S^3$ (3) is expected to be equivalent via the large $N$ duality
\cite{OV}, based on \cite{GV}.
This connection is based on the fact that the resolved conifold is the
resolution of singularities of the conifold $xy = zw$ in $\C^4$, while
$T^*S^3$ is its deformation.
Finally the topological string theory on $T^*S^3$ is equivalent to the
Chern-Simons theory \cite{WittenCS}. The equivalence is given by
comparing perturbative expansions of both theories.

Those equivalences were originally found for unrefined theories. Since
the Chern-Simons theory has a refinement by Khovanov-Rozansky, it
leads to a refinement of the topological string theory on $T^*S^3$ and
the resolved conifold \cite{GSV}. The computation in \cite{AS,AS2} is
based on this refined topological string theory, though the role of
the $S^1$-symmetry is clarified at the first time there.

On the other hand, Nekrasov's partition function has two parameters
$q$, $t$ corresponding to the $\C^*\times\C^*$-action on $\R^4 =
\C^2$. This observation leads to introduce a refinement of the theory
of topological vertex, which is a powerful tool to compute the
topological string theory on toric Calabi-Yau 3-folds, such as the
resolved conifold \cite{AK,IKV}. The computation via the refined
topological vertex nicely matches with the refined Chern-Simons theory
\cite{GIKV}. The $S$-matrix in \cite{AS} coincides with one found in
this context, up to the change of the base on the quantum Hilbert
space \cite{IK}.

The refined topological string theory is supposed to be compute
motivic Donaldson-Thomas invariants, which appear naturally in
\cite{KontS}, for the resolved conifold.
It is not clear, at least to the author, what the refined theory
computes for $T^*S^3$ in mathematics. Even if it becomes clear, all
equivalences in (1),(2),(3),(4) are difficult to justify
mathematically in full generality. In some special cases, they can be
checked by computing invariants separately and compare the formulas.
For example, the geometric engineering between (1) and (2) is checked
for the case corresponding to the invariant for $S^3$ without knots
\cite{MMNS} (see also \cite{Sz}).

Our main result express the large $N$ limit of the
$S$ matrix in terms of certain Euler characteristics of tensor
products of two copies of universal sheaves over Hilbert schemes.
The derivation is more or less straightforward from the explicit form
in \cite[(5.12)]{AS} and the relation between Macdonald polynomials
and Hilbert schemes \cite{Haiman} except that we use a version of
Cherednik-Macdonald-Mehta identity by Garsia, Haiman and Tesler
\cite{GHT}.
We remark that we still identify the quantum Hilbert space with the
equivariant $K$ group, but change the fixed point base to another base
given basically by Procesi bundles.

As an application, we find an explicit formula for the Euler
characteristics of a single universal sheaf, applied arbitrary Schur
functor (\corref{cor:empty}). This includes a formula of Scala for the
case of exterior powers, but general cases seem to be new to the best
of the author's knowledge.

Much more still remain to be done. One of the most important questions
is how to understand products of $S$, $T$ matrices in this
framework.
The multiplication is given by the convolution product in the
$K$-theory, but we cannot literally multiply the large $N$ limit of
$S$ and $T$, as they have infinite sizes.
We only see that the refined Chern-Simons theory gives a
function in $q$, $t$ for a fixed $N$ and varying $k$. But we do not
see what happens when we change also $N$. See \remsref{rem:sum}. 
After this problem will be solved, we next compute products, for example,
to check \cite[Conjectures~7,8]{ORS}, where the answer is given by a
single cohomology group, not products of several. The latter question
is probably harder than the former one.

The paper is organized as follows. The first section is preliminaries
for modified Macdonald polynomials and Hilbert schemes. 
Modified Macdonald polynomials are obtained from integral forms of
original Macdonald polynomials by applying the plethystic substitution.
These are naturally connected with Hilbert schemes in Haiman's theory.
In Section 2 we study specialization of Macdonald polynomials at $q^k
t^N = 1$.
In Section 3 we review the $S$ and $T$ matrices for the refined theory
and rewrite them in terms of modified Macdonald polynomials.
In Section 4 we use preliminary results in Section 1 to the formula in
Section 3 to get a geometric interpretation of the large $N$ limit of
the $S$-matrix.
In two appendices we collect notation on partitions and symmetric
polynomials.

\subsection*{Acknowledgments}

A preliminary version of this paper was written while the author was
enjoying the hospitality of the Institute for Advanced Study in 2007.
At that time the papers \cite{AS,AS2} did not appear yet, and the
motivation was to understand \cite{GIKV}.
The author thanks Sergei Gukov for discussion since then. He also
thanks Mina Aganagic for answering his questions on \cite{AS}, and
Shintaro Yanagida for his careful reading.

\section{Preliminaries}

\subsection{Plethystic substitution}\label{subsec:pleth}

In the main body of the paper we will need a $N^{\mathrm{th}}$ root of
${q}$ and the square root of ${t}$, but we write the function field
$\Q(q,t)$ instead of $\Q(q^{1/N},\sqrt{t})$ for brevity.

Let $R = \Q(q,t)[[x_1,x_2,\dots]]$ be the formal power series ring
in variables $x_1$, $x_2$, \dots with coefficients in
$\Q(q,t)$.

We define a $\lambda$-ring structure on $R$ by
\begin{equation*}
  \psi_k(q) = q^k, \quad \psi_k(t) = t^k, \quad \psi_k(x_i) = x_i^k,
\end{equation*}
where $\psi_k$ is the $k^{\mathrm{th}}$ Adams operation, which is a
ring homomorphism of $R$.

Let $\Lambda$ be the ring of symmetric functions.
We have a {\it plethysm} $\circ\colon \Lambda\times R\to R$ such
that $(-\circ A)\colon \Lambda\to R$ is a ring homomorphism for any
$A\in R$ and $(p_k\circ -)\colon R\to R$ is the Adams operation
$\psi_k$.
Given $A\in R$ and a symmetric
polynomial $f$, we denote $f\circ A$ by $f[A]$.
It is called {\it plethystic substitution\/} of $A$ into $f$.
In practice, the plethysm is determined by $p_k\circ$ as it is a ring
homomorphism in the first variable.
See \cite[Chapter 1]{Haglund}.

For example, if $X = x_1 + x_2 + \dots$, then $f[X] =
f(x_1,x_2,\dots)$, where $f$ in the left hand side is considered as an
element of $\Lambda$, and the right hand is considered as an element
of $R$.
We often denote the right hand side also by $f$, without specifying
either $f$ is in $\Lambda$, $R$ or $\Q(q,t)\otimes \Lambda$.
We use this capital letter $X$ throughout the paper. If
$y_1,y_2,\dots$ are another set of variables, we use $Y$ for
$y_1+y_2+\cdots$.
Various operations are compactly expressed by plethystic
substitutions. For example, if $f$ has degree $n$, then $f[q X] =
q^{n} f$.

More generally the plethysm is defined for any $\lambda$-ring $R$ in
the same way. Later we will use the plethysm when $R$ is the
Grothendieck ring $K(Y)$ of vector bundles over a variety $Y$. The
$\lambda$-ring structure is given by tensor products. For example, the
operation corresponding to the $k^{\mathrm{th}}$ elementary symmetric
function $e_k$ is the exterior product $\Wedge^k$.

Let us return back to our $R = \Q(q,t)[[x_1,x_2,\dots]]$ and compute
$f[-X]$. We have
\begin{equation*}
   p_k[-X] = \psi_k(-X) = - \psi_k(X) = - p_k.
\end{equation*}
\begin{NB}
  As $\psi_k$ is a ring homomorphism, $0 = \psi_k(-X) + \psi_k(X)$.
\end{NB}%
This is different from $p_k(-X) = (-1)^k p_k(X)$. Therefore if we want
to replace variables by their negatives in plethystic brackets, we
use the $\epsilon$-symbol, i.e.,
\begin{equation*}
  p_k[\epsilon X] = (-1)^k p_k.
\end{equation*}

Let $\omega$ be the ring involution of $\Lambda$ sending $e_n$ to
$h_n$, or equivalently $p_k$ to $(-1)^{k-1} p_k$.
Then $p_k[-X] = -p_k = (-1)^k \omega p_k = \omega p_k[\epsilon X]$.
Therefore $(\omega f)[X] = f[-\epsilon X]$ for a symmetric polynomial
$f$.

We also use $f[X/(1-q)]$ later. We have
\begin{equation*}
   p_k[\frac{X}{1-q}] = \frac{x_1^k + x_2^k + \dots}{1 - q^k}
   = \frac{p_k}{1 - q^k}.
\end{equation*}
It has inverse $g\mapsto g[X(1-q)]$. In fact,
$p_k[X(1-q)] = (x_1^k + x_2^k + \dots)(1-q^k)
= p_k(1-q^k)$.

Let
\begin{equation*}
  \Omega(x) = \prod_i \frac1{1-x_i}
  = \exp\left(\sum_{k=1}^\infty \frac{p_k(x)}k\right)
  = \sum_{k=0}^\infty h_k(x).
\end{equation*}
Then
\begin{equation*}
  \Omega[XY \frac{1-t}{1-q}] = \prod_{k=1}^\infty
  \exp\left(\frac1k \frac{1-t^k}{1-q^k} p_k(x) p_k(y)
    \right).
\end{equation*}
This is the kernel of the inner product $\langle\ ,\ \rangle_{q,t}$ in
\cite[VI]{Mac} given by $\langle f,g \rangle_{q,t} = \langle f,
g[\frac{1-q}{1-t}X]\rangle$.

We also consider
\begin{equation*}
  \tilde\Omega(x) = \omega\Omega(x) =
  \exp\left(\sum_{k=1}^\infty \frac{(-1)^{k-1}p_k(x)}{k}
  \right) 
  = \sum_{k=0}^\infty e_k(x)
  = \prod_i (1 + x_i).
\end{equation*}

\begin{NB}
Let us record a formula
\begin{equation*}
  \Omega[X+Y] = \Omega[X]\Omega[Y].
\end{equation*}
This implies that
\begin{equation*}
  h_n[X+Y] = \sum_{k+l=n} h_k(x) h_l(y).
\end{equation*}
\end{NB}

The following expression appears quite often:
\begin{equation}
  \label{eq:OmegaB}
  \frac1{\Omega[u B_\lambda(q,t)]}
  = \prod_{s\in\lambda} ( 1 - u q^{a'(s)} t^{l'(s)}).
\end{equation}
See \secref{sec:part} for the definition of $B_\lambda(q,t)$ and other
notations related to partitions.

\begin{NB}
  A little more nontrivial computation:
  \begin{equation*}
    \omega f[\frac{X}{1-t}]
    = f[-\frac{\epsilon X}{1-t}]
    = f[\frac{\epsilon t^{-1} X}{1- t^{-1}}]
    = (-t)^d f[\frac{X}{1- t^{-1}}]
  \end{equation*}
if $f$ is of degree $d$.
\end{NB}

\subsection{Modified Macdonald polynomials}


Let $P_\lambda(x;q,t)$ be the Macdonald polynomial for a partition
$\lambda$ as in \cite[VI]{Mac}. Its integral form is defined as
\begin{equation*}
  J_\lambda(x;q,t) \defeq c_\lambda(q,t) P_\lambda(x;q,t),
\end{equation*}
where $c_\lambda(q,t)$ is as in \secref{sec:part}.

Following works of Garsia-Haiman and their collaborators, we introduce
the modified Macdonald polynomial by
\begin{equation*}
   \tilde H_\mu(x;q,t) \defeq t^{n(\mu)} J_\mu\left[
     \frac{X}{1-t^{-1}};q,1/t
   \right].
\end{equation*}
\begin{NB}
  We have
  \begin{equation*}
    \tilde H_\mu[(1-t^{-1})X;q,t] = t^{n(\mu)} J_\mu(x;q,t^{-1}).
  \end{equation*}
We replace $t$ by $t^{-1}$ to get
\begin{equation*}
    \tilde H_\mu[(1-t)X;q,t^{-1}] = t^{-n(\mu)} J_\mu(x;q,t).
\end{equation*}
Therefore
\begin{equation}\label{eq:inverse}
   P_\mu(x;q,t) = \frac{t^{n(\mu)}}{c_\mu(q,t)}
   \tilde H_\mu[(1-t)X; q,t^{-1}].
\end{equation}
\end{NB}%
This subsection is devoted for the review of various properties of
$\tilde H_\mu(x;q,t)$ together with an introduction of a fundamental
operator $\nabla$. The reference \cite{Haiman2} is a nice survey
article on this topic and other things, which will be recalled later
in this section.

The Kostka-Macdonald coefficient $\tilde K_{\lambda\mu}(q,t)$ is
defined as
\begin{equation}
  \label{eq:K}
  \tilde H_\mu(x;q,t) = \sum_\lambda \tilde K_{\lambda\mu}(q,t) s_\lambda.
\end{equation}
It is related to one in \cite[VI (8.11)]{Mac} by $\tilde
K_{\lambda\mu}(q,t) = t^{n(\mu)} K_{\lambda\mu}(q,t^{-1})$, as
$S_\lambda(x;t) = s_\lambda[(1-t)X]$. Macdonald positivity conjecture
says $\tilde K_{\lambda\mu}(q,t)\in \Z_{\ge 0}[q,t]$, which was proved
by Haiman \cite{Haiman}. See the next subsection for a little more
detail.

The modified Macdonald polynomials are characterized by the properties
\begin{enumerate}
\item $\tilde H_\mu(x;q,t)\in \Q(q,t)\{ s_\lambda[X/(1-q)] \mid
  \lambda\ge \mu\}$,
\item $\tilde H_\mu(x;q,t)\in \Q(q,t)\{ s_\lambda[X/(1-t)] \mid
  \lambda\ge \mu^t\}$,
\item $\tilde H_\mu[1;q,t] = 1$.
\end{enumerate}
See \cite[3.5.2]{Haiman2}.

We have $P_\mu(x;q^{-1},t^{-1}) = P_\mu(x;q,t)$
\cite[VI(4.14)(iv)]{Mac}. It implies
\begin{equation}\label{eq:omegaH}
   \tilde H_\mu(x;q,t) 
   = \omega \tilde H_\mu(x;q^{-1},t^{-1}) q^{n(\mu^t)} t^{n(\mu)}.
\end{equation}
(See \cite[3.5.12]{Haiman2}.) Let us give a proof for the sake of the
reader. We start with
\begin{equation*}
  P_\mu[\frac{X}{1-t^{-1}};q,t^{-1}]
  = P_\mu[\frac{X}{1-t^{-1}};q^{-1},t]
  = P_\mu[\frac{-tX}{1-t};q^{-1},t].
\end{equation*}
Therefore
\begin{equation*}
   \tilde H_\mu(x;q,t) = t^{2n(\mu)} \frac{c_\mu(q,t^{-1})}{c_\mu(q^{-1},t)} 
   \tilde H_\mu[-tX;q^{-1},t^{-1}]
   = t^{2n(\mu)} \frac{c_\mu(q,t^{-1})}{c_\mu(q^{-1},t)} 
   \omega \tilde H_\mu[\epsilon tX;q^{-1},t^{-1}].
\end{equation*}
Now the assertion follows from
\begin{equation}\label{eq:temp}
    \frac{c_\mu(q^{-1},t)}{c_\mu(q,t^{-1})}
    = \prod_{s\in \mu} \frac{1-q^{-a(s)} t^{l(s)+1}}{1 - q^{a(s)} t^{-(l(s)+1)}}
    = q^{-n(\mu^t)} t^{n(\mu)} (-t)^{|\mu|},
\end{equation}
\begin{NB}
For the record
  \begin{equation*}
    \frac{c'_\lambda(q,t^{-1})}{c'_\lambda(q^{-1},t)}
    = \prod_{s\in\lambda} \frac{1-q^{a(s)+1} t^{-l(s)}}{1 - q^{-(a(s)+1)} t^{l(s)}}
    = t^{-n(\lambda)} q^{n(\lambda^t)} (-q)^{|\lambda|}
  \end{equation*}
and
\begin{equation}\label{eq:q-1}
  \frac{c_\lambda(q^{-1},t) c'_\lambda(q,t^{-1})}
  {c_\lambda(q,t^{-1})c'_\lambda(q^{-1},t)}
  = (qt)^{|\lambda|}.
\end{equation}
\end{NB}%
and $\tilde H_\mu[\epsilon tX;q^{-1},t^{-1}] = (-t)^{|\mu|} \tilde
H_\mu(x;q^{-1},t^{-1})$.
\begin{NB}
We give another proof for a record:
\begin{equation*}
  \begin{split}
  & \omega \tilde H_\mu(x;q^{-1},t^{-1})
  = t^{-n(\mu)} J_\mu\left[-\epsilon\frac{X}{1-t};q^{-1},t \right]
  = t^{-n(\mu)} c_\mu(q^{-1},t)
  P_\mu\left[-\epsilon \frac{X}{1-t};q^{-1},t \right]
\\
  =\; & t^{-n(\mu)} c_\mu(q^{-1},t)
  P_\mu\left[-\epsilon \frac{X}{1-t};q,t^{-1} \right]
  = t^{-n(\mu)} c_\mu(q^{-1},t)
  P_\mu\left[\epsilon t^{-1}\frac{X}{1-t^{-1}};q,t^{-1} \right]
\\
  =\; & (-t)^{-|\mu|} t^{-2n(\mu)}
  \frac{c_\mu(q^{-1},t)}{c_\mu(q,t^{-1})}
    H_\mu(x;q,t).
  \end{split}
\end{equation*}
Then we use \eqref{eq:temp}.
\end{NB}%
Other formulas appearing below can be also derived from the
corresponding formulas for the original Macdonald polynomials. We
leave the proofs as exercises for the reader hereafter.

Let $\omega_{q,t}$ be the algebra endomorphism of
$\Q(q,t)\otimes \Lambda$ defined by
\begin{equation*}
  \omega_{q,t} p_k = (-1)^{k-1} \frac{1-q^k}{1-t^k} p_k.
\end{equation*}
We have
\begin{equation*}
  \omega_{q,t} f = \omega f[\frac{1-q}{1-t} X]
  = f[-\epsilon \frac{1-q}{1-t} X].
\end{equation*}
It is known \cite[VI, (5.1)]{Mac} that
\begin{equation*}
  \omega_{q,t} P_\lambda(x;q,t) 
  = \frac{c_{\lambda^t}(t,q)}{c_{\lambda^t}'(t,q)} P_{\lambda^t}(x;t,q)
  = \frac{c_{\lambda}'(q,t)}{c_{\lambda}(q,t)} P_{\lambda^t}(x;t,q).
\end{equation*}
This is rewritten as
\begin{equation}\label{eq:transpose}
  \tilde H_{\mu^t}(x;q,t) = \tilde H_\mu(x;t,q).
\end{equation}
(See \cite[3.5.11]{Haiman2}.)
\begin{NB}
  Let us give a proof. Using \eqref{eq:inverse}, we get
  \begin{equation*}
    \omega_{q,t} P_\mu(x;q,t) 
    = \frac{t^{n(\mu)}}{c_{\mu}(q,t)} \omega \tilde H_\mu[(1-q)X;q,t^{-1}]
    = \frac{q^{n(\mu^t)}}{c_{\mu}(q,t)}\tilde H_\mu[(1-q)X;q^{-1},t],
  \end{equation*}
where we have used \eqref{eq:omegaH}.
On the other hand we have
\begin{equation*}
  \frac{c_{\mu^t}(t,q)}{c_{\mu^t}'(t,q)} P_{\mu^t}(x;t,q) 
  = \frac{q^{n(\mu^t)}}{c_{\mu^t}'(t,q)}
  \tilde H_{\mu^t}[(1-q)X;t,q^{-1}]
  =
  \frac{q^{n(\mu^t)}}{c_\mu(q,t)}
  \tilde H_{\mu^t}[(1-q)X;t,q^{-1}].
\end{equation*}
Replacing $(1-q)X$ by $X$, and then $q^{-1}$ by $q$, we get the
assertion.
\end{NB}

The Cauchy formula for the usual Macdonald polynomials \cite[VI(4.13)
and (6.19)]{Mac} is
\(
    \Omega[XY \frac{1-t}{1-q}]
    = \sum_{\mu} P_\mu(X;q,t) Q_\mu(Y;q,t).
\)
It is equivalent to
\begin{equation}\label{eq:Cauchy}
  \tilde\Omega\left[
    \frac{XY}{(1-q)(1-t)}
  \right]
  = \sum_{\mu}
  \frac{q^{-n(\mu^t)} t^{-n(\mu)} \tilde H_\mu(X;q,t) \tilde H_\mu(Y;q,t)}
  {c_\mu(q^{-1},t) c_\mu'(q,t^{-1})}.
\end{equation}
\begin{NB}
Note
\begin{equation*}
  c_\mu(q^{-1},t) c_\mu'(q,t^{-1}) =
  \prod_{s\in \mu} 
    \left(1 - t^{-l_{\mu}(s)} q^{a_{\mu}(s)+1}\right)
      \left(1 - t^{l_{\mu}(s)+1}q^{-a_{\mu}(s)}\right).
  \end{equation*}
\end{NB}%
\begin{NB}
Let us give a proof.
Recall \cite[VI(4.13) and (6.19)]{Mac}:
  \begin{equation*}
    \Omega[XY \frac{1-t}{1-q}]
    = \sum_{\mu} P_\mu(X;q,t) Q_\mu(Y;q,t)
    = \sum_\mu \frac{J_\mu(X;q,t) J_\mu(Y;q,t)}
    {c_\mu(q,t) c_\mu'(q,t)}.
  \end{equation*}
  We replace $t$ by $t^{-1}$ and then $X$ by $X/(1-t^{-1})$, $Y$ by
  $Y/(1-t^{-1})$ to get
  \begin{equation*}
    \begin{split}
    & \Omega[XY \frac1{(1-t^{-1})(1-q)}]
    = \sum_\mu \frac{J_\mu[\frac{X}{1-t^{-1}};q,t^{-1}]
        J_\mu[\frac{Y}{1-t^{-1}};q,t^{-1}]}
    {c_\mu(q,t^{-1}) c_\mu'(q,t^{-1})}
\\
   =\; &
   \sum_\mu \frac{H_\mu(x;q,t) H_\mu(Y;q,t) q^{-n(\mu^t)} t^{-n(\mu)}
     (-t)^{|\mu|}}
    {c_\mu(q^{-1},t) c_\mu'(q,t^{-1})},
    \end{split}
  \end{equation*}
  where we have used \eqref{eq:temp}. On the other hand, the left hand
  side is
  \begin{equation*}
    \Omega[\frac{-tXY}{(1-t)(1-q)}]
    = \omega\Omega(\frac{\epsilon tXY}{(1-t)(1-q)}].
  \end{equation*}
Replacing $X$ by $\epsilon t^{-1}X$, we get the assertion.
\end{NB}%
We apply $\omega$ to get
\begin{equation}\label{eq:Cauchy2}
  \Omega\left[
    \frac{XY}{(1-q)(1-t)}
  \right]
  = \sum_{\mu}
  \frac{q^{-n(\mu^t)} t^{-n(\mu)} \omega \tilde H_\mu(X;q,t) \tilde H_\mu(Y;q,t)}
  {c_\mu(q^{-1},t) c_\mu'(q,t^{-1})}.
\end{equation}
\begin{NB}
Note
\begin{equation*}
  (\omega\Omega)[XY]
  = \exp\left(
    \sum_{k\ge 1}(-1)^{k-1} \frac{p_k(X) p_k(Y)}k
    \right)
  = \omega^X \Omega[XY] = \omega^Y\Omega[XY],
\end{equation*}
where the superscripts $X$, $Y$ denote the operations are applied to
$X$-variables or $Y$-variables. To get the above formula, we have used
$\omega^X$.

This can be checked directly as follows. The above says
\begin{equation*}
  \Omega\left[
    \frac{XY}{(1-q)(1-t)}
  \right]
  = \sum_{\mu}
  \frac{\tilde H_\mu(X;q^{-1},t^{-1}) \tilde H_\mu(Y;q,t)}
  {c_\mu(q^{-1},t) c_\mu'(q,t^{-1})}
\end{equation*}
by \eqref{eq:omegaH}. We replace $q$, $t$ by $q^{-1}$, $t^{-1}$
respectively.
Using \eqref{eq:q-1} we have
\begin{equation*}
  \Omega\left[
    \frac{qt XY}{(1-q)(1-t)}
  \right]
  =
  \Omega\left[
    \frac{XY}{(1-q^{-1})(1-t^{-1})}
  \right]
  = \sum_{\mu}
  \frac{(qt)^{|\mu|}\tilde H_\mu(X;q,t) \tilde H_\mu(Y;q^{-1},t^{-1})}
  {c_\mu(q^{-1},t) c_\mu'(q,t^{-1})}.
\end{equation*}
This is equivalent to
\begin{equation*}
  \Omega\left[
    \frac{XY}{(1-q)(1-t)}
  \right]
  = \sum_{\mu}
  \frac{\tilde H_\mu(X;q,t) \tilde H_\mu(Y;q^{-1},t^{-1})}
  {c_\mu(q^{-1},t) c_\mu'(q,t^{-1})}
\end{equation*}
\end{NB}%

We define another inner product by
\begin{equation}
  \langle f,g\rangle_* = \langle f, \omega g[(1-q)(1-t) X]\rangle
  = (-t)^{\deg g} \langle f, g[(1-q)(1-t^{-1}) X]\rangle,
\end{equation}
where $\deg g$ denotes the degree of $g$.
\begin{NB}
  \begin{equation*}
    \langle p_\lambda, p_\mu\rangle_*
    = \delta_{\lambda\mu} (-1)^{|\lambda|-l(\lambda)} z_\lambda 
    p_\lambda[(1 - q)(1-t)].
  \end{equation*}
Note also that
\begin{equation*}
  \begin{split}
  (\omega p_\lambda)[(1-q)(1-t) X] 
  &= (-1)^{|\lambda|- l(\lambda)} p_\lambda[(1-q)(1-t)X]
  = (-1)^{|\lambda|- l(\lambda)} \prod_i (1-q^{\lambda_i})(1-t^{\lambda_i})
  p_\lambda
\\
  &= \omega (p_\lambda[(1-q)(1-t)X]).
  \end{split}
\end{equation*}
Therefore we do not need to care the order of the plethystic
substitution and $\omega$.

We also note
\begin{equation*}
  \left\langle s_\lambda\left[\frac{X}{1-q}\right],
  s_\mu\left[\frac{X}{1-t^{-1}}\right]\right\rangle_*
  = (-t)^{|\lambda|}\delta_{\lambda\mu}.
\end{equation*}
\end{NB}%
We have a relation
\begin{equation}\label{eq:rel}
    \langle f, g\rangle_{q,t}
  = \langle f[\frac{X}{1-t}], \omega g[\frac{X}{1-t}]\rangle_*
  = (-t)^{-\deg g} \langle f[\frac{X}{1-t}], g[\frac{X}{1-t^{-1}}]\rangle_*.
\end{equation}
\begin{NB}
  \begin{equation*}
      \mathrm{LHS} = \langle f, g[\frac{1-q}{1-t}X]\rangle
      = \langle f[\frac{X}{1-t}], g[(1-q)X]\rangle
      = (-t)^{-d} \langle f[\frac{X}{1-t}],
      g[\frac{X}{1-t^{-1}}]\rangle_*
      = \mathrm{RHS}.
  \end{equation*}
\end{NB}%
The Cauchy formula \eqref{eq:Cauchy} is equivalent to
\begin{equation}\label{eq:Hinner}
  \langle \tilde H_\lambda, \tilde H_\mu\rangle_*
  = \delta_{\lambda\mu} q^{n(\mu^t)} t^{n(\mu)} c_\mu(q^{-1},t) c'_\mu(q,t^{-1}).
\end{equation}
See \cite[(73)]{Haiman2}. This can be directly derived from the
formula for inner products of Macdonald polynomials (\cite[VI(6.19)]{Mac}):
\(
  \langle P_\lambda,P_\mu\rangle_{q,t} = \delta_{\lambda\mu}
  \nicefrac{c_\lambda'(q,t)}{c_\lambda(q,t)},
\)
and the property
\(
  \langle f,g\rangle_{q,t^{-1}} = \langle f[-tX], g\rangle_{q,t}.
\)

\begin{NB}
Let us compute this more directly.
  \begin{equation*}
    \langle \tilde H_\lambda, \tilde H_\mu\rangle_*
    = \langle \tilde H_\lambda[(1-t){X};q,t],
    \omega\tilde H_\mu[(1-t){X};q,t]\rangle_{q,t}
    = q^{n(\mu^t)} t^{n(\mu)}
    \langle \tilde H_{\lambda}[{(1-t)}X;q,t],
    \tilde H_{\mu}[{(1-t)}X;q^{-1},t^{-1}]\rangle_{q,t}.
  \end{equation*}
By \eqref{eq:inverse} we have
\begin{equation*}
  \tilde H_{\mu}[(1-t){X};q^{-1},t^{-1}]
  = t^{-n(\mu)} J_{\mu}(x;q^{-1},t) 
  = t^{-n(\mu)} c_{\mu}(q^{-1},t) P_{\mu}(x;q^{-1},t).
\end{equation*}
On the other hand, we have
\begin{equation*}
  \tilde H_{\lambda}[(1-t){X};q,t]
  = \tilde H_{\lambda}[(1-t^{-1}){(-tX)};q,t]
  = t^{n(\lambda)} c_{\lambda}(q, t^{-1}) P_{\lambda}[-tX; q,t^{-1}].
\end{equation*}
Therefore
\begin{equation*}
  \langle \tilde H_\lambda, \tilde H_\mu\rangle_*
  = q^{n(\mu^t)} t^{n(\lambda)} c_\lambda(q,t^{-1}) c_\mu(q^{-1},t)
  \langle P_\lambda[-tX;q,t^{-1}], P_\mu(x;q,t^{-1})\rangle_{q,t},
\end{equation*}
where we have used  $P_\mu(x;q^{-1},t) = P_\mu(x;q,t^{-1})$.

The Cauchy formula for the Macdonald polynomials is
\begin{equation*}
  \Omega[\frac{1-t}{1-q} XY] = \sum_\mu
  P_\mu(x;q,t) Q_\mu(y;q,t).
\end{equation*}
The left hand side is equal to
\begin{equation*}
  \Omega[\frac{1-t^{-1}}{1-q} (-tX)Y].
\end{equation*}
Therefore replacing $X$ by $-t^{-1}X$, and then $t$ by $t^{-1}$ to get
\begin{equation*}
  \Omega[\frac{1-t}{1-q} XY]
  = \sum_\mu P_\mu[-tX;q,t^{-1}] Q_\mu[Y;q,t^{-1}].
\end{equation*}
Since the left hand side is the same as the original kernel, this means
\begin{equation*}
  \langle P_\lambda[-tX;q,t^{-1}], Q_\mu[X;q,t^{-1}]\rangle_{q,t}
  = \delta_{\lambda\mu},
\end{equation*}
or equivalently
\begin{equation*}
  \langle P_\lambda[-tX;q,t^{-1}], P_\mu[X;q,t^{-1}]\rangle_{q,t}
  = \delta_{\lambda\mu} \frac{c_\mu'(q,t^{-1})}{c_\mu(q,t^{-1})}.
\end{equation*}
Substituting this to back, we find
\begin{equation*}
  \langle \tilde H_\lambda, \tilde H_\mu\rangle_* = \delta_{\lambda\mu} 
  q^{n(\mu^t)} t^{n(\mu)} c_\mu(q^{-1},t) c_\mu(q,t^{-1})
  \frac{c_\mu'(q,t^{-1})}{c_\mu(q,t^{-1})}
  = \delta_{\lambda\mu} 
  q^{n(\mu^t)} t^{n(\mu)} c_\mu(q^{-1},t) c_\mu'(q,t^{-1}).
\end{equation*}

\begin{NB2}
  It is little strange that we need to the Cauchy formula during the
  proof. So let us give an equivalent, but slightly different
  looking explanation. Note
  \begin{equation*}
    \begin{split}
  \langle p_\lambda, p_\mu\rangle_{q,t^{-1}}
  &= \delta_{\lambda\mu} z_\lambda p_\lambda[\frac{1-q}{1-t^{-1}}]
  = \delta_{\lambda\mu} z_\lambda p_\lambda[-\frac{t(1-q)}{1-t}]
  = (-1)^{|\lambda|-l(\lambda)} t^{|\lambda|}
  \delta_{\lambda\mu} z_\lambda p_\lambda[\frac{1-q}{1-t}]
\\
  &= \langle p_\lambda[-tX], p_\mu\rangle_{q,t}.
    \end{split}
  \end{equation*}
Therefore we have $\langle f, g\rangle_{q,t^{-1}}
= \langle f[-tX], g\rangle_{q,t}$. Therefore
\begin{equation*}
  \langle P_\lambda[-tX;q,t^{-1}], Q_\mu(X;q,t^{-1})\rangle_{q,t}
  = \langle P_\lambda(x;q,t^{-1}), Q_\mu(x;q,t^{-1})\rangle_{q,t^{-1}}
  = \delta_{\lambda\mu}.
\end{equation*}
\end{NB2}
\end{NB}%

\begin{NB}
Let us also note that
\begin{equation*}
  \langle J_\lambda, J_\lambda\rangle_{q,t}
  = c_\lambda(q,t)^2 \langle P_\lambda, P_\lambda\rangle_{q,t}
  = c_\lambda(q,t) c_\lambda'(q,t).
\end{equation*}
\end{NB}

We define an operator $\nabla$ (see \cite[3.5.5]{Haiman2}) on
$\Lambda\otimes\Q(q,t)$ by
\begin{equation}\label{eq:nabla}
  \nabla \tilde H_\mu = q^{n(\mu^t)} t^{n(\mu)} \tilde H_\mu.
\end{equation}

This operator can be generalized as follows. Let $f$ be a symmetric
polynomial. We define $\Delta_f$ by
\begin{equation*}
  \Delta_f \tilde H_\mu = f[B_\mu(q,t)] \tilde H_\mu.
\end{equation*}
We have $\nabla = \Delta_{e_n}$ on the degree $n$ part.

\subsection{Hilbert schemes of points on the plane
and their torus fixed points}\label{subsec:Hilb}

We fix terminology on Hilbert schemes, in particular, recall weights
of tangent and tautological vector bundles at fixed points of a torus
action. 
%
Our convention for the torus action is the same as in \cite{MR2199008}
(see especially \cite[Remark~4.4]{MR2199008} for the action on
functions) instead of \cite{MR2183121,GNY2}, where $t_1$, $t_2$ are
replaced by their inverse. See also \cite{Lecture}.

Set $X = \C^2$ and let $X^{[n]}$ denote the Hilbert scheme of $n$
points in $X$. Set-theoretically it is nothing but the space of ideals
$I$ in $\C[x,y]$ such that $\dim \C[x,y]/I = n$.
We have the Hilbert-Chow morphism $\pi\colon X^{[n]}\to S^n X$, where
$S^n X$ is the $n^{\mathrm{th}}$ symmetric power of $X$.

We have a $\C^*\times\C^*$-action on $X$ and the induced action on
$X^{[n]}$. To make compatible with \cite{Haiman} we make it
\begin{equation*}
  (x,y) \mapsto (t x, q y).
\end{equation*}
So $t_1$, $t_2$ in \cite{MR2199008} are $t$, $q$ here.

We consider the equivariant $K$-group $K_{\C^*\times\C^*}(X^{[n]})$ of
$X^{[n]}$ with respect to the $\C^*\times \C^*$-action.
It is the Grothendieck group of the abelian category of
$\C^*\times\C^*$-equivariant vector bundles over $X^{[n]}$.
As $X^{[n]}$ is smooth, it is naturally isomorphic also to the
Grothendieck group of the category of equivariant sheaves.
It is a module over the representation ring $R(\C^*\times\C^*)\cong
\Z[q^{\pm 1}, t^{\pm 1}]$.

Tensor products $\otimes$ of equivariant vector bundles are
well-defined on $K_{\C^*\times\C^*}(X^{[n]})$. It is denoted also by
$\otimes$.

Later we will use the derived category
$D^b_{\C^*\times\C^*}(\operatorname{Coh}X^{[n]})$ of equivariant
coherent sheaves on $X^{[n]}$. The equivariant $K$-group
$K_{\C^*\times\C^*}(X^{[n]})$ is isomorphic to the free abelian group
generated by the objects of
$D^b_{\C^*\times\C^*}(\operatorname{Coh}X^{[n]})$ modulo relation
given by distinguished triangles.

The $\C^*\times\C^*$-fixed points in $X^{[n]}$ are monomial ideals in
$\C[x,y]$, and hence are parametrized by partitions $\mu$ of $n$.
Let $I_\mu\in X^{[n]}$ be the fixed point corresponding to
$\mu$. Then the $\C^*\times\C^*$-character of the alternating
sum of exterior powers of the cotangent space
$T^*_{I_\mu} X^{[n]}$ is given by
\begin{equation}\label{eq:HilbTang}
   \ch \Wedge_{-1}T_{I_\mu}^* X^{[n]}
   = \prod_{s\in \mu} 
    \left(1 - t^{-l_{\mu}(s)} q^{a_{\mu}(s)+1}\right)
      \left(1 - t^{l_{\mu}(s)+1} q^{-a_{\mu}(s)}\right)
   = c_\mu(q^{-1},t) c_\mu'(q,t^{-1}).
\end{equation}
Note that this expression appears in \eqref{eq:Cauchy2}.

Over the Hilbert scheme $X^{[n]}$ we have the {\it tautological
  bundle\/} $\shfO^{[n]}$ whose fiber at $I$ is given by
\begin{equation*}
    \shfO^{[n]}_I = \C[x,y]/I.
\end{equation*}
The fiber over the fixed point $I_\mu$ is a
$\C^*\times\C^*$-module whose character is
\begin{equation*}
  \ch \shfO^{[n]}_{I_\mu}
  = \sum_{s\in \mu} q^{a'(s)} t^{l'(s)} = B_\mu(q,t).
\end{equation*}
The generating function of the exterior powers of the tautological
bundle at the point $I_\mu$ satisfies
\begin{equation}\label{eq:taut}
  \ch \Wedge_{-u} \shfO^{[n]}_{I_\mu}
  = \prod_{s\in\mu}( 1 - u q^{a'(s)} t^{l'(s)})
  = \frac1{\Omega[u B_\mu(q,t)]}.
\end{equation}

Let us consider its determinant line bundle of $\shfO^{[n]}$:
\begin{equation*}
    \mathcal L = \det \shfO^{[n]} = \Wedge^n \shfO^{[n]}.
\end{equation*}
We have
\begin{equation}\label{eq:CS}
   \ch \mathcal L_{I_\mu}
   = \prod_{s\in\mu} q^{a'(s)} t^{l'(s)} 
   = q^{n(\mu^t)} t^{n(\mu)}.
\end{equation}

Let $\calE$ be the universal (rank $1$ torsion free) sheaf over
$X\times X^{[n]}$. We consider it as a $K$-group class and denote it
by the same letter also. Let $\calE_{(0,I_\mu)}$ be its fiber at
$(0,I_\mu)$. By \cite[(4.1)]{MR2183121} its character is given by
\begin{equation}\label{eq:universal}
  \ch \calE_{(0,I_\mu)} = 1 - (1-q)(1-t) B_\mu(q,t) = A_\mu(q,t).
\end{equation}
Since $\ch \C[x,y] = 1/(1-q)(1-t)$, we have
\begin{equation*}
    \ch \calE_{(0,I_\mu)} = 1 - \frac{\ch \shfO^{[n]}_{I_\mu}}{\ch \C[x,y]}.
\end{equation*}

We later use slant products of $\calE$ with classes on $\C^2$, for
example $\uE$, $\calE/[\C^2]$. They are classes in the equivariant
$K$-group, defined as
\begin{equation*}
  \uE = q_{2*} (\mathcal E\otimes q_1^* [0]), \qquad
  \calE/[\C^2] = q_{2*} (\mathcal E) = \frac1{(1-q)(1-t)} \uE,
\end{equation*}
where $[0]$ is the class in $K_{\C^*\times\C^*}(X)$ of the origin, and
$q_1$, $q_2$ are the projections of $X\times X^{[n]}$ to the first and
second factors.
Note that $\calE/[\C^2]$ is a {\it localized\/} class in
$K_{\C^*\times\C^*}(X)\otimes_{R(\C^*\times\C^*)} \Q(q,t)$, not in
$K_{\C^*\times\C^*}(X)$. It is because $q_{2}$ is not proper, and the
push-forward homomorphism $q_{2*}$ is defined by formally using
Atiyah-Bott-Lefschetz fixed point formula as above. More examples of
pushforward homomorphisms for non proper morphisms appear later.

\subsection{Haiman's works}\label{subsec:Haiman}

Much deeper connection between Hilbert schemes $X^{[n]}$ and
symmetric polynomials are known through Haiman's works
\cite{Haiman,Haiman-vanish}. Let us briefly recall the main result.

Let $X_n$ be the reduced fiber product
\begin{equation}\label{eq:diagram}
  \begin{CD}
    X_n @>{f}>> X^{n}= \C^{2n}
\\
   @V{\rho}VV @VVV
\\
   X^{[n]} @>>> S^n X = \C^{2n}/S_n.
  \end{CD}
\end{equation}
Then one of main results in \cite{Haiman} says that $X^{[n]}$ is
isomorphic to the Hilbert schemes of $S_n$-invariant ideals in
$\C^{2n}$ so that $X_n$ is identified with its universal family.

As an application of this result together with the result by
Bridgeland-King-Reid \cite{BKR}, Haiman proved

\begin{Theorem}[\cite{Haiman-vanish}]
The functor
\begin{equation*}
  \Phi \defeq {\mathbf R}f_*\circ \rho^* \colon
  D^b(\operatorname{Coh}X^{[n]})
  \to D^b_{S_n}(\operatorname{Coh}X^{n}),
\end{equation*}
is an equivalence of categories, where
$D^b(\operatorname{Coh}X^{[n]})$ is the derived category of coherent
sheaves on $X^{[n]}$ and $D^b_{S_n}(\operatorname{Coh}X^{n})$ is the
derived category of $S_n$-equivariant coherent sheaves on $X^n$.
\end{Theorem}

The latter derived category is identified with the derived category of
bounded complexes of finitely generated $S_n$-equivariant
$\C[x,y]^{\otimes n} = \C[x_1,y_1,\dots, x_n,y_n]$-modules.

The above equivalence holds also for equivariant derived categories
with respect to the $\C^*\times\C^*$-action.

Let $\cP = \rho_*\shfO_{X_n}$. This is a vector bundle of rank $n! =
|S_n|$ endowed with fiberwise $S_n$-action. We call it the {\it
  Procesi bundle}. Then the functor $\Phi$ is identified with
\begin{equation*}
  \Phi = {\mathbf R}\Gamma_{X^{[n]}}(\cP\otimes -).
\end{equation*}

Taking equivariant Grothendieck group, we have a natural isomorphism
\begin{equation*}
  \Phi\colon K_{\C^*\times\C^*}(X^{[n]}) \xrightarrow{\cong}
  K_{S_n\times \C^*\times\C^*}(X^n).
\end{equation*}

We further compose a homomorphism given by the push-forward
homomorphism for $X^n\to\operatorname{pt}$. Since this morphism is not
a proper, it is not well-defined in the level of the ordinary
Grothendieck group. However, it is well-defined if we take the {\it
  localized\/} Grothendieck group $K_{S_n\times
  \C^*\times\C^*}(\operatorname{pt})\otimes_{\Z[q^{\pm 1},t^{\pm 1}]}\Q(q,t)$.
This can be done by using Atiyah-Bott-Lefschetz formula as there is
only finitely many (in fact only one) torus fixed point in $X^n$ as in
the previous subsection. It can be also checked by using a standard
argument on Hilbert series. See \cite[\S3]{Haiman-vanish} for more
detail.
Recall that the Grothendieck group $R(S_n)$ of representations of
$S_n$ is identified with the degree $n$ part of $\Lambda$ via the
Frobenius characteristic map. See \secref{sec:symm}. Therefore we
combine the above Hilbert series map with Frobenius characteristic
map, we get
\begin{equation*}
  K_{S_n\times \C^*\times\C^*}(X^n) \to \Q(q,t)\otimes \Lambda.
\end{equation*}
The composed map $K_{\C^*\times\C^*}(X^{[n]})\to \Q(q,t)\otimes
\Lambda$ is also denoted by $\Phi$. From its construction, we can
identify the image of $\Phi$:
\begin{Corollary}[\protect{\cite[Prop.~5.4.6]{Haiman2}}]\label{cor:K}
  The functor $\Phi$ induces an isomorphism 
  \begin{equation*}
    K_{\C^*\times\C^*}(X^{[n]}) \cong \{ f\in \Q(q,t)\otimes \Lambda^n \mid
    f[(1-q)(1-t)X]\in \Z[q^{\pm 1},t^{\pm 1}]\otimes \Lambda^n \}.
  \end{equation*}
  The subgroup of classes supported on the punctual Hilbert scheme
  $\pi^{-1}(0)$ is isomorphic to $\Z[q^{\pm 1},t^{\pm 1}]\otimes \Lambda^n$.
\end{Corollary}

Let $I_\mu$ be the torus fixed point as in the previous
subsection. Then $\Phi(I_\mu)$ is the fiber of $\cP$ at $I_\mu$. Then
one of the main result in \cite{Haiman} says that this is equal to
$\tilde H_\mu$.

We decompose $\cP = \bigoplus_\lambda V^\lambda\otimes \cP_\lambda$
where $V^\lambda$ is the irreducible representation of $S_n$
corresponding to the partition $\lambda$. Then the fiber of
$\cP_\lambda$ at $I_\mu$, as a $\C^*\times\C^*$-module, has character
$\tilde K_{\lambda\mu}(q,t)$. This proves the positivity conjecture of
$\tilde K_{\lambda\mu}$, mentioned in the previous subsection.

There is a natural $\Q(q,t)$-valued symmetric bilinear form $(\ ,\ )$
on $K_{\C^*\times\C^*}(X^{[n]})$ induced from the composite of the
functor
\begin{equation*}
  \mathbf R\Gamma_{X^{[n]}}(-\otimes^{\mathbf L}-)
\end{equation*}
and the character $\ch\colon R(\C^*\times\C^*)\to \Z[q^{\pm 1}, t^{\pm 1}]$.
The Koszul resolution gives us the formula $(I_\mu,I_\nu) =
\delta_{\mu\nu} \ch \Wedge_{-1} T^*_{I_\mu}X^{[n]}$. Comparing
\eqref{eq:HilbTang} with \eqref{eq:Hinner}, we find that $\langle\ ,\
\rangle_*$ has an extra factor $q^{n(\mu^t)} t^{n(\mu)}$. This is
given by the operator $\nabla$. Therefore
\begin{equation}\label{eq:geom}
  (-,-) = \langle\nabla^{-1}\Phi(-),\Phi(-)\rangle_*.
\end{equation}
See \cite[Cor.~5.4.7]{Haiman2}.

The operator $\nabla$ itself has a geometric interpretation. We have
an endofunctor on $D^b(\operatorname{Coh}X^{[n]})$ given by $- \otimes
\mathcal L$, where $\mathcal L$ is the determinant of the tautological
bundle as in the previous subsection. The character of its fiber at
$I_\mu$ is given by $q^{n(\mu^t)} t^{n(\mu)}$. This means that
\begin{equation*}
  \Phi(-\otimes\mathcal L) = \nabla \Phi(-).
\end{equation*}
See \cite[Cor.~5.4.9]{Haiman2}.

Similarly the operator $\Delta_f$ corresponds to the operator
$-\otimes f(\shfO^{[n]})$, where $f(\shfO^{[n]})$ is the plethystic
substitution of the $K$-theory class of $\shfO^{[n]}$ to $f$.

We have
\begin{equation*}
  \Phi(\mathcal P_{\lambda^t}\otimes \mathcal L^{*})
  = s_\lambda\left[\frac{X}{(1-q)(1-t)}\right].
\end{equation*}
See \cite[(97)]{Haiman2}.

\begin{Remark}
  There is also a natural inner product on
  $K_{S_n\times\C^*\times\C^*}(X^n)$ defined in the same way as for
  $(\ ,\ )$.
But the equivalence $\Phi$ does {\it not\/} respect inner products.
In fact, the natural inner product on $K_{S_n\times\C^*\times\C^*}(X^n)$
is $\langle\ ,\ \rangle_*$. This follows from
\begin{equation*}
  \left\langle
    s_\lambda\left[\frac{X}{(1-t)(1-q)}\right],
    s_\mu\right\rangle_*
  = 
  \delta_{\lambda^t\mu},
\end{equation*}
and the observation that $s_\lambda[{X}/{(1-t)(1-q)}]$ corresponds to
$V^\lambda\otimes \shfO_{X^n}$ and $s_\mu$ to the tensor product of
$V^\mu$ and the skyscraper sheaf at the origin.
The discrepancy of inner products comes from that $\Phi$ satisfies
\begin{equation*}
  \Hom_{D^b_{S_n}(\operatorname{Coh} X^n)}(\Phi(a), \Phi(b))
  \cong \Hom_{D^b(\operatorname{Coh} X^{[n]})}(a, b),
\end{equation*}
as it is a categorical equivalence. We note $\mathbf
R\Hom_{D^b(\operatorname{Coh}X^{[n]})}(a,b) = \mathbf
R\Gamma_{X^{[n]}}(a^*\otimes b)$, and a similar identity for
$\Hom_{D^b_{S_n}(\operatorname{Coh} X^n)}$, where $a^*$ is the
Grothendieck-Verdier duality.
Then we use $\mathcal P_\lambda^* = \mathcal
P_{\lambda^t}\otimes\mathcal L$ (see \cite[(96)]{Haiman2}) to see
that $(-,-) = \langle \nabla^{-1}(-), -\rangle_*$.
\begin{NB}

A little more detail.
We have
\begin{equation*}
  \delta_{\lambda^t\mu}
  = \Hom\left(s_{\lambda^t}\left[\frac{X}{(1-t)(1-q)}\right],s_\mu\right).
\end{equation*}
This is equal to
\begin{equation*}
  \Hom\left(\Phi^{-1}(s_{\lambda^t}\left[\frac{X}{(1-t)(1-q)}\right]),
    \Phi^{-1}(s_\mu)\right),
\end{equation*}
as we have just explained. As 
\begin{equation*}
  \Phi^{-1}(s_{\lambda^t}\left[\frac{X}{(1-t)(1-q)}\right])
  = \mathcal P_{\lambda}\otimes\mathcal L^*
  = \mathcal P_{\lambda^t}^*,
\end{equation*}
this corresponds to
\begin{equation*}
  \mathbf R\Gamma_{X^{[n]}}\left(\mathcal P_{\lambda^t}\otimes
    \Phi^{-1}(s_\mu)\right)
  =
  \left(\nabla s_\lambda\left[\frac{X}{(1-t)(1-q)}\right],
    s_\mu\right).
\end{equation*}
\end{NB}%
\end{Remark}

\subsection{Classes supported on a lagrangian subvariety}
\label{subsec:lag}

We need a variant of \corref{cor:K} later. Let $L_x^{[n]}$ be the
inverse image of the symmetric product of the $x$-axis $S^n(\{ y =
0\})$ under $\pi$. We also introduce a similar variety $L_y^{[n]}$
exchanging the role of $x$ and $y$. These are known to be lagrangian
with respect the natural symplectic form on $X^{[n]}$ induced from
that on $X$. (See \cite[\S9.3]{Lecture}.)

Then \corref{cor:K} in this setting yields
\begin{equation}\label{eq:KL}
  K_{\C^*\times\C^*}(L_x^{[n]}) \cong
  \{ f\in \Q(q,t)\otimes\Lambda^n \mid
  f[(1-t)X]\in\Z[q^{\pm 1}, t^{\pm 1}]\otimes \Lambda^n \}.
\end{equation}

In particular, we have an integral base of
$K_{\C^*\times\C^*}(L_x^{[n]})$ given by
\begin{equation*}
  \{ \Phi^{-1}(s_\lambda[{X}/{(1-t)}]) \}_\lambda.
\end{equation*}
Classes $\Phi^{-1}(s_\lambda[{X}/{(1-t)}])$ appeared in Haiman's
work. We briefly recall the result. See \cite[\S5.4.2]{Haiman2} for
more detail.

Recall the commutative diagram \eqref{eq:diagram}. Let $V({\mathbf y})
\subset X_n$ be the zero locus $y_1 = y_2 = \cdots = y_n = 0$, where
$y_i$ is the pull-back of the coordinate function on $X^n$ under
$f$. It is a complete intersection in $X_n$ and $\rho$ maps
$V({\mathbf y})$ to $L^{[n]}_x$. Then
\begin{equation*}
  \Phi^{-1}(s_\lambda[{X}/{(1-t)}])
  = \mathcal L^{*}\otimes \Hom_{S_n}(V^{\lambda^t},\rho_*\shfO_{V({\mathbf y})}),
\end{equation*}
where $V^{\lambda^t}$ is the irreducible representation of $S_n$
corresponding to the transpose partition $\lambda^t$ of $\lambda$, and
$\rho_*\shfO_{V({\mathbf y})}$ has an induced $S_n$-action from that
on $X_n$.

\subsection{Geometric interpretation of the Cauchy formula}

Let us give a few applications of the Cauchy formula for the Macdonald
polynomials under the above identifications of various expressions in
terms of Hilbert schemes.
These calculation will be prototypes of our interpretation of refined
Hopf link HOMFLY invariants as Hilbert series on Hilbert schemes.

We take degree $n$ in \eqref{eq:Cauchy2} and substitute $X = Y =
1$. Then we get
\begin{equation}\label{eq:specializedCauchy2}
  h_n[1/(1-q)(1-t)]
  = \sum_{|\mu|=n}
  \frac{1}
  {c_\mu(q^{-1},t) c_\mu'(q,t^{-1})},
\end{equation}
as $\omega \tilde H_\mu[1;q,t] = t^{n(\mu)} q^{n(\mu^t)}$ by
\eqref{eq:omegaH}. 
\begin{NB}
If we substitute $X = Y = 1$ into \eqref{eq:Cauchy} instead, we get
\begin{equation}\label{eq:specializedCauchy}
  e_n[{1}/{(1-q)(1-t)}]
  = \sum_{|\mu|=n}
  \frac{q^{-n(\mu^t)} t^{-n(\mu)}}
  {\prod_{s\in \mu} 
    \left(1 - t^{-l_{\mu}(s)} q^{a_{\mu}(s)+1}\right)
      \left(1 - t^{l_{\mu}(s)+1}q^{-a_{\mu}(s)}\right)}.
\end{equation}
\end{NB}%
By \eqref{eq:HilbTang} the right hand side is $(\ch
\Wedge_{-1} T^*_{I_\mu}X^{[n]})^{-1}$. By the Atiyah-Bott-Lefschetz
fixed point formula, we have
\begin{equation*}
  \sum_{\mu}\frac1{\ch \Wedge_{-1} T^*_{I_\mu}X^{[n]}}
  = \sum_d (-1)^d \ch H^d (X^{[n]}, \shfO_{X^{[n]}}).
\end{equation*}
(See \cite[\S3]{Haiman-vanish} why this is true even though $X^{[n]}$
is noncompact.) 
We safely confuse the derived functor with the corresponding
homomorphism on the $K$-theory, and we denote the right hand side by
\begin{equation*}
  \ch {\mathbf R}\Gamma_{X^{[n]}}(\shfO_{X^{[n]}})
\end{equation*}
hereafter.

On the other hand, the left hand side of \eqref{eq:specializedCauchy2}
is the character of $S^n \C[x,y]$, the $n^{\mathrm{th}}$ symmetric
power of $\C[x,y]$. It is equal to the character of the coordinate
ring of $S^n X$. Therefore \eqref{eq:specializedCauchy2} implies
\begin{equation}\label{eq:S3}
  \ch \Gamma_{S^n X}(\shfO_{S^nX})
  = \ch {\mathbf R}\Gamma_{X^{[n]}}(\shfO_{X^{[n]}}).
\end{equation}
In fact, this equality is a simple consequence of the fact $\mathbf
R\pi_* \shfO_{X^{[n]}} = \shfO_{S^n X}$, which follows as $S^n X$ only
has a rational singularity.

Note also that this geometric argument gives a stronger assertion that
\begin{equation*}
  H^d (X^{[n]}, \shfO_{X^{[n]}})
  =
  \begin{cases}
    S^n \C[x,y] & \text{if $d=0$},
\\
    0 & \text{otherwise}.    
  \end{cases}
\end{equation*}

This \eqref{eq:S3} can be generalized as follows. Let us put $X=1$, $Y=1-u$ in
\eqref{eq:Cauchy2}. The left hand side is
\begin{equation*}
  \begin{split}
  & h_n\left[\frac{1-u}{(1-q)(1-t)}\right]
  = \sum_{k=0}^n (-u)^k h_{n-k}\left[\frac1{(1-q)(1-t)}\right]
  e_k\left[\frac1{(1-q)(1-t)}\right]
\\
  =\; &
  \sum_{k=0}^n (-u)^k \ch (S^{n-k} \C[x,y]\otimes
  \Wedge^k \C[x,y]).
  \end{split}
\end{equation*}
On the other hand, the right hand side is
\begin{equation*}
  \sum_{|\mu|=n} \frac{\tilde H_\mu[1-u;q,t]}{c_\mu(q^{-1},t) c_{\mu}'(q,t^{-1})}.
\end{equation*}
From \cite[VI(6.11')]{Mac} or \cite[(3.5.20)]{Haiman2} we have $\tilde
H_\mu[1-u;q,t] = \Omega[uB_\mu(q,t)]^{-1}$. (See also
\thmref{thm:symmetry} below.)
Therefore this is equal to
\begin{equation*}
  \sum_{|\mu|=n} \frac{\ch \Wedge_{-u}\shfO^{[n]}_{I_\mu}}
  {\ch \Wedge_{-1} T^*_{I_\mu} X^{[n]}}
  = \sum_{d} (-1)^d \ch H^d(X^{[n]}, \Wedge_{-u} \shfO^{[n]})
\end{equation*}
thanks to \eqref{eq:taut}. Again a stronger geometric result is known
(see \cite[5.2.1]{Scala})
\begin{equation*}
    H^d (X^{[n]}, \Wedge^k \shfO^{[n]})
  =
  \begin{cases}
    S^{n-k} \C[x,y] \otimes \Wedge^k \C[x,y] & \text{if $d=0$},
\\
    0 & \text{otherwise}.    
  \end{cases}
\end{equation*}

The module $\Wedge^k \C[x,y]$ is infinite dimensional as a vector
space. We can remedy it by replacing $\shfO^{[n]}$ by $\uE$ as follows.
Let us take the simplest case $k=1$ as an example.
As $\uE = \mathcal O_{X^{[n]}} - (1-q)(1-t) \shfO^{[n]}$, 
we have
\begin{equation*}
  \mathbf R\Gamma_{X^{[n]}}(\uE)
  = S^n\C[x,y] - S^{n-1}\C[x,y]
\end{equation*}
from the above consideration, where we have used
\begin{equation*}
  (1-q)(1-t) \C[x,y] \cong \C.
\end{equation*}
This is nicely packed into generating functions:
\begin{equation*}
  \sum_{n=0}^\infty
  u^n \ch \mathbf R\Gamma_{X^{[n]}}(\uE)
  = (1 - u)
  \sum_{n=0}^\infty u^n \ch \mathbf R\Gamma_{X^{[n]}}(\shfO_{X^{[n]}})
  = (1-u)\Omega[\frac{u}{(1-q)(1-t)}].
\end{equation*}
Thus we get a polynomial $1-u$ if we divide the answer by
\(
\Omega[\frac{u}{(1-q)(1-t)}].
\)
This is an example of more general phenomena, which we will see later
in \thmref{thm:HilbHilb} and \corref{cor:empty}.

\begin{NB}
  The convention in \cite{AS} and ours are slightly different: $t$ is
  replaced by $t^{-1}$, hence
  \begin{equation*}
    \Omega[\frac{u}{(1-q)(1-t^{-1})}] = \exp\left(
      \sum_{k=1}^\infty \frac{u^k}{k(1-q^{k})(1-t^{-k})}
    \right)
    = \exp\left(
      - \sum_{k=1}^\infty \frac{(u\sqrt{t}/\sqrt{q})^k}
      {k(q^{k/2}-q^{-k/2})(t^{k/2}-t^{-k/2})}
    \right).
  \end{equation*}
  is considered \cite[(6.1)]{AS}. The numerator $u\sqrt{t}/\sqrt{q}$
  is denoted by $\lambda$ in \cite{AS}, and later specialized as
  \(
    \lambda = t^N \sqrt{t}/\sqrt{q}.
  \)
  This means that $u = t^{N}$ or $u = t^{-N}$ for our $t$.
\end{NB}

\section{Specialization}

The original Chern-Simons gauge theory depends on two integers $N$ and
$k$, namely the rank of the gauge group $SU(N)$ and the level $k$. In
the corresponding quantum group approach, the Lie algebra is
$\mathfrak{su}(N)$ (or its complexification) and the quantum parameter
$q$ is taken as $q=\exp(\nicefrac{2\pi i}{k+N})$. In the refined
theory, another parameter $t$ appears so that we identify $q$, $t$
with parameters in Macdonald polynomials. But they are not
arbitrary, and given by
\begin{equation*}
  q = \exp(\nicefrac{2\pi i}{k+\beta N}), \quad
  t = \exp(\nicefrac{2\pi i\beta}{k+\beta N}),
\end{equation*}
where $\beta$ is a parameter for the refinement. The original
Chern-Simons theory corresponds to $\beta = 1$, and it is sometimes
better to think that $\beta$ is an integer parameter so that $q$ is a
root of unity. However we analytically continue $\beta$ to the
complex plane, as in \cite{AS}. (A similar thing already appears in
\cite[VI\S9]{Mac}.) Then we only impose the constraint
\begin{equation}
  \label{eq:qt}
  q^k t^N = 1.
\end{equation}
We study the behavior of Macdonald polynomials specialized at
\eqref{eq:qt}.

In fact, this specialization will be performed in two steps. In the
first step, we choose $N$. It corresponds to the projection from
$\Lambda$ to the ring $\Z[x_1,\dots, x_N]^{S_N}$ of symmetric
polynomials with finitely many variables. This is an elementary
step. In the second step, we choose $k$ and impose the condition
\eqref{eq:qt}.
In particular, we emphasize the role of $\Omega[uB_\mu(q,t)]^{-1}$
appeared naturally in \eqref{eq:taut}.

Note also that we {\it break\/} the symmetry between $q$, $t$ and $k$,
$N$ here. We first choose $t$ and $N$ to make a specialization. We can
also choose $q$ and $k$ first. But the answer will be different.

\subsection{Finite $N$}

There is a natural inner product, suitable for Macdonald polynomials
with a finite number of variables $x = (x_1,\dots, x_N)$. There are
two versions, but we use $\langle\ ,\ \rangle''_N$ \cite[VI,
(9.8)]{Mac}. It is defined by
\begin{equation}\label{eq:Ninner}
  \langle P_\lambda,P_\mu\rangle''_N = \delta_{\lambda\mu}
  \langle P_\lambda, P_\lambda\rangle_{q,t}
  \prod_{s\in\lambda} \frac{1 - q^{a'(s)} t^{N-l'(s)}}
  {1 - q^{a'(s)+1} t^{N-l'(s)-1}}
  = \delta_{\lambda\mu}\frac{c'_\lambda(q,t)}{c_\lambda(q,t)}
  \frac{\Omega[t^{N-1}q B_\lambda(q,t^{-1})]}
  {\Omega[t^{N}B_\lambda(q,t^{-1})]}
  \begin{NB}
  = \delta_{\lambda\mu}
  \prod_{s\in \lambda}\frac{1-t^{l(s)}q^{a(s)+1}}
  {1-t^{l(s)+1}q^{a(s)}}
  \frac{1- t^{N-l'(s)}q^{a'(s)}}{1-t^{N-l'(s)-1}q^{a'(s)+1}}
  \end{NB}%
  .
\end{equation}
Note that the right hand side vanishes if there is a box $s\in
\lambda$ such that $a'(s) = 0$, $l'(s) = N$, i.e., $l(\lambda) >
N$. This is compatible with the property \cite[VI (4.10)]{Mac} that
$P_\lambda(x_1,\dots,x_N;q,t) = 0$ if $l(\lambda) > N$. In fact, this
inner product has more natural definition on symmetric functions with
$N$ variables \cite[VI \S9]{Mac}. So the radical of $\langle\ ,\
\rangle''_N$ is the kernel of the projection $\Lambda\to
\Z[x_1,\dots,x_N]^{S_N}$.
This inner product was used in \cite[Appendix A]{AS}.
\begin{NB}
There was a typo in \cite{AS} in the earlier version.
\end{NB}%

\begin{NB}
This is an analog of $\langle\ ,\ \rangle_{q,t}$. So we consider
an analog for $\langle\ ,\ \rangle_*$. Let us define
$\langle\ ,\ \rangle_{*,N}$ by
\begin{equation*}
  \langle f, g \rangle_{*,N} = \langle
  f[(1-t)X], \omega g[(1-t)X]\rangle''_N.
\end{equation*}
Then
\begin{equation*}
  \begin{split}
  \langle \tilde H_\lambda, \tilde H_\mu \rangle_{*,N}
  = q^{n(\mu^t)} t^{n(\lambda)} c_\lambda(q,t^{-1}) c_\mu(q^{-1},t)
  \langle P_\lambda[-tX;q,t^{-1}], P_\mu(x;q,t^{-1})\rangle''_N
  = 
  \end{split}
\end{equation*}
But this seems difficult to compute................
\end{NB}

Let us recall a useful lemma, which will be used frequently later.

\begin{Lemma}[\protect{\cite[VI(4.17)]{Mac}}]\label{lem:cut}
If $l(\mu)=N$, we have
\[
  P_\mu(x_1,\dots,x_N;q,t)
  = x_1\cdots x_N P_{\mu-(1,\dots,1)}(x_1,\dots, x_N;q,t),
\]
where $\mu - (1,\dots,1)$ denotes a partition $(\mu_1 - 1,\dots,\mu_N-1)$.
\end{Lemma}

\subsection{Specialization at $q^k t^N = 1$}\label{subsec:spec2}

We first specify what the specialization at $q^k t^N = 1$ means.
A naive guess suggests to consider $\Z[q^{\pm 1}, t^{\pm 1}]/(q^k t^N
= 1)$, but it has a trouble when $k$ and $N$ are not coprime. We
follow an idea in \cite{FJMM}.
\begin{NB}
  In fact, it seems that there is no problem at all. We have
  \begin{equation*}
    0 = q^k v^N - 1 = (q^{k/m}v^{N/m} - 1)(q^{k/m}t^{N/m} - \omega )
    \cdots (q^{k/m} t^{N/m} - \omega^{m-1}).
  \end{equation*}
But all expression we will have are forms of $1 - q^a t^b$.
\end{NB}%
Let $m = \operatorname{gcd}(k,N)$. We choose a primitive
$m^{\mathrm{th}}$ root of unity $\omega$. We also choose
$\omega_1\in\C$ such that $\omega_1^{N/m} = \omega$. We introduce a
new variable $v$ and set
\begin{equation*}
  q = v^{{N}}, \quad t = v^{{-k}}\omega_1.
\end{equation*}
\begin{NB}
Original version
\begin{equation*}
  t = v^{{k}/m}, \quad q = v^{-{N}/m} \omega_1.
\end{equation*}
\end{NB}%
For a $\Z[q^{\pm 1}, t^{\pm 1}]$-module $M$, we consider
$M\otimes_{\Z[q^\pm, t^\pm]}\C(v)$.
This procedure is what we mean by the specialization at $q^k t^N = 1$.
For example, we say $f\in M$ is nonzero at $q^k t^N = 1$ when
$f\otimes_{\Z[q^\pm, t^\pm]}\C(v)$ is nonzero. In this case $f^{-1}$
is well-defined in $\C(v)$. We say $f^{-1}$ is well-defined at $q^k
t^N = 1$.
\begin{NB}
We have
\begin{equation*}
  q^{1/N} = v,  \quad
  t^{-1/k} = v\omega_1^{-1/k}.
\end{equation*}
\end{NB}%

Strictly speaking, we should call this the specialization at $q^{k/m}
t^{N/m}
\begin{NB}
= v^{-Nk/m}\omega_1^{k/m} v^{kN/m}
\end{NB}%
= \omega$. But we do not for brevity. Anyway $\omega$ plays no
special role except the property $\omega^p = 1$ if and only if $p$ is
a multiple of $m$ in this paper.

\begin{NB}
  Later we use $u$, which is set $t^{-N}$. We have
  \begin{equation*}
    u = t^{N} = v^{{Nk}} = q^{-k}.
  \end{equation*}
\end{NB}

We consider
\[
   \Omega[t^NB_\lambda(q,t^{-1})]^{-1} 
   = \prod_{s\in\lambda} (1 - t^{N-l'(s)} q^{a'(s)}).
\]
We have the following simple observation.

\begin{Lemma}\label{lem:vanish}
  \textup{(1)} $\Omega[t^NB_\lambda(q,t^{-1})]^{-1}$ vanishes unless
  $l(\lambda)\le N$.

  \textup{(2)} Suppose $l(\lambda)\le N$. At $q^k t^N = 1$,
  $\Omega[t^NB_\lambda(q,t^{-1})]^{-1}$ vanishes if and only if
  $l(\lambda^t) > k$. Moreover the order of zero is $1$.
\end{Lemma}

\begin{proof}
  (1) This is observed already above.

  (2) If $l(\lambda^t) > k$, we consider the box $s$ in the first
  column and $(k+1)^{\mathrm{th}}$ row. Then $l'(s) = 0$, $a'(s) =
  k$. Therefore $t^{N-l'(s)} q^{a'(s)} = q^k t^N = 1$. The converse
  and the second assertion can be proved as in the proof of
  \lemref{lem:nonvanish} below. We leave them as exercises for the
  reader.
\end{proof}

\begin{NB}
  On the other hand,
  \begin{equation*}
     \Omega[t^{N-1}q B_\lambda(q,t^{-1})]^{-1} 
   = \prod_{s\in\lambda} (1 - t^{N-l'(s)-1} q^{a'(s)+1})
  \end{equation*}
never vanish.
\end{NB}%

Suppose $l(\lambda)\le N$. Then the corresponding Young diagram
decomposes uniquely into the upper part consisting of rows shorter
than $N-1$ and the lower part of rows with exactly length $N$. Let us
denote partition corresponding to the upper part by
$\overline{\lambda}$. See Figure~\ref{fig:2Young}.

\begin{Lemma}\label{lem:nonvanish}
  Suppose $l(\lambda)\le N$

  \textup{(1)} $c_\lambda'(q,t)$ is nonzero at $q^kt^N = 1$.

  \textup{(2)} $c_\lambda(q,t)$ is zero at $q^k t^N = 1$ if and only
  if $l(\lambda^t) > k$ and $l(\overline{\lambda}^t) \le k$. Moreover
  the order of zero is $1$.
\end{Lemma}

\begin{figure}[htbp]
  \centering
  \setlength{\unitlength}{1mm}
  \begin{picture}(100,65) \put(5,28){\makebox(4,4){$\lambda = $}}
    \put(20,10){\framebox(60,20){}}
    \put(30,38){\makebox(30,6){$\overline{\lambda}$ with
        $l(\overline{\lambda}) < N$}}
    \put(47,5){\vector(-1,0){27}}
    \put(53,5){\vector(1,0){27}}
    \put(48,3){\makebox(4,4){$N$}}
    \put(20,30){\line(0,1){30}}
    \put(20,60){\line(1,0){20}}
    \put(40,60){\line(0,-1){10}}
    \put(40,50){\line(1,0){20}}
    \put(60,50){\line(0,-1){5}}
    \put(60,45){\line(1,0){10}}
    \put(70,45){\line(0,-1){15}}
  \end{picture}
  \caption{Two Young diagrams}
  \label{fig:2Young}
\end{figure}
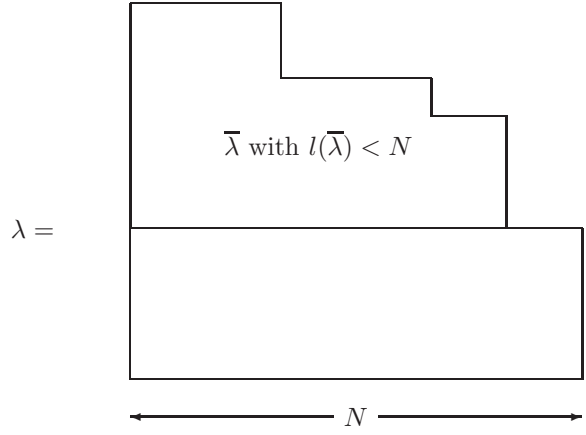

\begin{proof}
  We prove only (2). The assertion (1) is easier to prove.

  Recall $c_\lambda(q,t) = \prod_{s\in \lambda}(1-q^{a(s)}
  t^{l(s)+1})$. We have
  \begin{equation*}
    q^{a(s)} t^{l(s)+1} = \omega_1^{l(s)+1}
    v^{{N} a(s) - {k}(l(s)+1)}.
  \end{equation*}
  \begin{NB}
  \begin{equation*}
    q^{a(s)} t^{l(s)+1} = \omega_1^{a(s)}
    v^{\nicefrac{k}m(l(s)+1) - \nicefrac{N}m a(s)}.
  \end{equation*}
  \end{NB}%
  This is equal to $1$ if and only if 
  \begin{equation*}
    \frac{k}m (l(s)+1) = \frac{N}m a(s), \quad \omega_1^{l(s)+1} = 1.
  \end{equation*}
  Since $\nicefrac{k}m$ and $\nicefrac{N}m$ are coprime, there exists an
  integer $p$ such that $l(s)+1 = \nicefrac{Np}m$, $a(s) =
  \nicefrac{kp}m$. Therefore $\omega_1^{l(s)+1}
  \begin{NB}
  = \omega_1^{\nicefrac{Np}m}
  \end{NB}%
  = \omega^p$. Hence $p$ must be a multiple of $m$. On the other hand,
  $N\ge l(s)+1$ by our assumption $l(\lambda)\le N$. Therefore $p\le
  m$. Combining two conditions, we must have $p=0$ or $m$. The first
  case cannot happen as $l(s)+1 = 0$ is not possible. In the second
  case we have $l(s) = N-1$, $a(s) = k$. As $l(\lambda)\le N$, a box
  $s$ with $l(s)=N-1$ can occur only when $l(\lambda)=N$ and $s$ is in
  the first column. Then $a(s) = k$ means that $s$ is in the
  $(\lambda_1 -k)^{\mathrm{th}}$ row. Therefore we must have
  $l(\lambda^t) = \lambda_1 > k$. To have $l(s)=N-1$, this row must
  have length $N$, which means that it is not contained in
  $\overline{\lambda}$. It is equivalent to $l(\overline{\lambda}^t)\le k$.
  Converse direction is clear. The last assertion is clear from the proof.

  \begin{NB}
    Original version:

  Recall $c_\lambda(q,t) = \prod_{s\in \lambda}(1-q^{a(s)}
  t^{l(s)+1})$. We have
  \begin{equation*}
    q^{a(s)} t^{l(s)+1} = \omega_1^{a(s)}
    v^{{N} a(s) - {k}(l(s)+1)}.
  \end{equation*}
  \begin{NB2}
  \begin{equation*}
    q^{a(s)} t^{l(s)+1} = \omega_1^{a(s)}
    v^{\nicefrac{k}m(l(s)+1) - \nicefrac{N}m a(s)}.
  \end{equation*}
  \end{NB2}%
  This is equal to $1$ if and only if 
  \begin{equation*}
    \frac{k}m (l(s)+1) = \frac{N}m a(s), \quad \omega_1^{a(s)} = 1.
  \end{equation*}
  Since $\nicefrac{k}m$ and $\nicefrac{N}m$ are coprime, there exists an
  integer $p$ such that $l(s)+1 = \nicefrac{Np}m$, $a(s) =
  \nicefrac{kp}m$. Therefore $\omega_1^{a(s)}
  \begin{NB2}
  = \omega_1^{\nicefrac{kp}m}
  \end{NB2}%
  = \omega^p$. Hence $p$ must be a multiple of $m$. On the other hand,
  $N\ge l(s)+1$ by our assumption $l(\lambda)\le N$. Therefore $p\le
  m$. Combining two conditions, we must have $p=0$ or $m$. The first
  case cannot happen as $l(s)+1 = 0$ is not possible. In the second
  case we have $l(s) = N-1$, $a(s) = k$. As $l(\lambda)\le N$, a box
  $s$ with $l(s)=N-1$ can occur only when $l(\lambda)=N$ and $s$ is in
  the first column. Then $a(s) = k$ means that $s$ is in the
  $(\lambda_1 -k)^{\mathrm{th}}$ row. Therefore we must have
  $l(\lambda^t) = \lambda_1 > k$. To have $l(s)=N-1$, this row must
  have length $N$, which means that it is not contained in
  $\overline{\lambda}$. It is equivalent to $l(\overline{\lambda}^t)\le k$.
  Converse direction is clear. The last assertion is clear from the proof.
  \end{NB}

  \begin{NB}
    (1)
  For $c'_\lambda(q,t)$, we have
  \(
    q^{a(s)+1} t^{l(s)} = \omega_1^{a(s)+1}
    v^{\nicefrac{k}m l(s) - \nicefrac{N}m(a(s)+1)}.
  \)
  This is equal to $1$ if and only if
  \begin{equation*}
    \frac{k}m l(s) = \frac{N}m (a(s)+1), \quad \omega_1^{a(s)+1} = 1.
  \end{equation*}
  Therefore $l(s) = \nicefrac{Np}m$, $a(s)+1 =
  \nicefrac{kp}m$. Therefore $\omega_1^{a(s)+1} = \omega^p$. Thus $p$
  must be a multiple of $m$. But we must have $l(s)\le N-1$, hence
  $p=0$. But it is impossible as $a(s)+1 = 0$ cannot hold.
  \end{NB}%
\end{proof}

The above two vanishings are related as follows:
\begin{Lemma}
  \begin{equation*}
    \begin{split}
      {c_\lambda(q,t)}{\Omega[t^NB_\lambda(q,t^{-1})]} &=
      {c_{\overline{\lambda}}(q,t)}{\Omega[t^NB_{\overline{\lambda}}(q,t^{-1})]},
      \\
      {c_\lambda'(q,t)}{\Omega[t^{N-1}qB_\lambda(q,t^{-1})]} &=
      {c_{\overline{\lambda}}'(q,t)}{\Omega[t^{N-1}q
        B_{\overline{\lambda}}(q,t^{-1})]}.
    \end{split}
  \end{equation*}
\end{Lemma}

\begin{NB}
The equality in (2) means that the zero of $c_\lambda(q,t)$ cancels
with the pole of $\Omega[t^N B_\lambda(q,t^{-1})]$, as th right hand
side ${c_{\nu}(q,t)}{\Omega[t^NB_\nu(q,t^{-1})]}$ is nonzero.
\end{NB}%

\begin{proof}
  These assertions can be proved directly, but also follow from
  \eqref{eq:first} below and \lemref{lem:cut}: We have
  $P_\lambda(t^{\nicefrac{(N-1)}2},\dots,t^{\nicefrac{(1-N)}2};q,t) =
  P_{\overline{\lambda}}(t^{\nicefrac{(N-1)}2},\dots,t^{\nicefrac{(1-N)}2};q,t)$. And
  $2n(\lambda) - 2n(\overline{\lambda}) =
  (\lambda_1-\overline{\lambda}_1)N(N-1) =
  (N-1)(|\lambda|-|\overline{\lambda}|)$. Hence we get the first
  equality. The second equality follows from \eqref{eq:Ninner} and the
  fact $\langle P_\lambda,P_\lambda\rangle''_N = \langle
  P_{\overline{\lambda}}, P_{\overline{\lambda}}\rangle''_N$, which is
  a consequence of the definition in \cite[VI \S9]{Mac}.
\end{proof}

\begin{NB}
  These results probably related to \cite[Lemma~2.1]{FJMM}. They
  introduced a condition called the $(k,r,n)$-admissibility. In our
  notation, it reads as $\lambda$ is {\it admissible\/} if and only if
  \begin{equation*}
    \lambda_i - \lambda_{i+N-1} \ge k+1 
    \quad\text{for $i=1,\dots, l(\lambda)-N+1$}.
  \end{equation*}
  or equivalently
  \begin{equation*}
    \lambda_{j-N+1} - \lambda_{j} \ge k+1 
    \quad\text{for $j=N,\dots, l(\lambda)$}.
  \end{equation*}
\end{NB}

\begin{NB}
\subsection{}

  We will later consider
  \begin{equation*}
    \Omega[\frac{u}{(1-q)(1-t)}] = \exp\left(
      \sum_{k=1}^\infty \frac{u^k}{k(1-q^k)(1-t^k)}
    \right).
  \end{equation*}
  This is an element in $\Q(q,t)[[u]]$. But we want to substitute $u =
  t^{-N}$ later, so we need to expand this first with respect to $q$
  and $t$. We consider
  \begin{equation*}
    \Omega[\frac{u}{(1-q)(1-t)}]
    = \Omega[u(1+q+\dots)(1+t+\dots)]
    = \exp\left(
      \sum_{m,n=0}^\infty \sum_{l=1}^\infty \frac{u^l q^{ml} t^{nl}}l
      \right).
  \end{equation*}
  So the constant term with respect
  to $q$, $t$ is
\(
   \exp\left(\sum_{k=1}^\infty \frac{u^k}{k}\right)
   = \nicefrac1{1-u}.
\)
Therefore $\Omega[u/(1-q)(1-t)]$ is in $\Q[[u]][[q,t]]$. So it is not
clear why we can substitute $u = t^{-N}$ in this expansion.

If we formally substitute $u = t^{-N}$, we get
  \begin{equation*}
    \exp\left(
      \sum_{m,n=0}^\infty \sum_{l=1}^\infty \frac{q^{ml} t^{(n-N)l}}l
      \right).
  \end{equation*}
How we interpret this ?

\begin{NB2}
  It may be understood as
  \begin{equation*}
    \Omega[\frac{u}{(1-q)(1-t)}]
    = \Omega[\frac{uq^{-1}t^{-1}}{(1-q^{-1})(1-t^{-1})}]
    \in \Q[[u]][[q^{-1}, t^{-1}]].
  \end{equation*}
  But this is different from the usual understanding of
  $\Omega[u/(1-q)(1-t)]$ as $\sum u^n \ch \mathbf
  R\Gamma_{X^{[n]}}(\shfO_{X^{[n]}})$, so we must be careful.
\end{NB2}

We recall the computation of \cite[\S6.2]{AS} on
$S_{\emptyset\emptyset}$.

Suppose that $t = q^{-\beta}$ with an integer $\beta$. Then
\begin{equation*}
  \begin{split}
  & \exp\left[\sum_{k=1}^\infty \frac{t^{-Nk}}{k}\frac1{(1-q^k)(1-t^k)}\right]
\\
=\; &
  \exp\left[-\sum_{k=1}^\infty \frac{t^{-Nk}}k
    \frac{(1-q^{\beta k})t^{k}}{(1-q^k)(1-t^k)^2}\right]
\\
=\; &
  \exp\left[-\sum_{m=0}^{\beta-1} \sum_{k=1}^\infty \frac{t^{-Nk}q^{mk}}k
    \frac{t^k}{(1-t^k)^2}\right]
\\
=\; &
  \exp\left[-\sum_{m=0}^{\beta-1} \sum_{k=1}^\infty \frac{t^{-Nk}q^{mk}}k
    \sum_{i=1}^\infty i t^{ki}\right]
\\
=\; &
  \exp\left[-\sum_{m=0}^{\beta-1} \sum_{i,k=1}^\infty \frac{i}k
    t^{-(N-i)k}q^{mk} \right]
\\
=\; &
   \prod_{m=0}^{\beta-1} \prod_{i=1}^\infty (1 - t^{-(N-i)}q^m)^i.
  \end{split}
\end{equation*}
In \cite[(6.2)]{AS} it is written that this is the `large $N$ limit of'
\begin{equation*}
  S_{\emptyset\emptyset}(q,t^{-1}) = 
     \prod_{m=0}^{\beta-1} \prod_{i=1}^{N-1} (1 - t^{-(N-i)}q^m)^i.
\end{equation*}
Where the difference
\begin{equation*}
     \prod_{m=0}^{\beta-1} \prod_{i=N}^{\infty} (1 - t^{-(N-i)}q^m)^i
     = \prod_{m=0}^{\beta-1} \prod_{i=0}^{\infty} (1 - t^{i}q^m)^{i+N}.
\end{equation*}
go ?

If we replace the product by $\prod_{i=1}^{N-1}$ and then follow the
computation back, we get
\begin{equation*}
  \begin{split}
   & \prod_{m=0}^{\beta-1} \prod_{i=1}^{N-1} (1 - t^{-(N-i)}q^m)^i
\\
=\; &
  \exp\left[-\sum_{m=0}^{\beta-1} \sum_{k=1}^\infty \frac{t^{-Nk}q^{mk}}k
    \sum_{i=1}^{N-1} i t^{ki}\right]
\\
=\; &
  \exp\left[-\sum_{m=0}^{\beta-1} \sum_{k=1}^\infty \frac{t^{-Nk}q^{mk}}k
    \frac{t^k(1-N t^{k(N-1)}+(N-1)t^{kN})}{(1-t^k)^2}\right]
\\
=\; &
  \exp\left[-\sum_{k=1}^\infty \frac{t^{-Nk}}k
    \frac{(1-q^{\beta k})t^{k}(1-N t^{k(N-1)}+(N-1)t^{kN})}
    {(1-q^k)(1-t^k)^2}\right]
\\
=\; & \exp\left[\sum_{k=1}^\infty \frac{t^{-Nk}}{k}\frac
    {(1-N t^{k(N-1)}+(N-1)t^{kN})}
    {(1-q^k)(1-t^k)}\right].
  \end{split}
\end{equation*}
\begin{NB2}
Starting from here, we find
\begin{equation*}
  \begin{split}
  & \Omega[\frac{t^{-N}-Nt^{-1} + N-1}{(1-q)(1-t)}]
  = \Omega[\frac1{1-q}(t^{-1}+ t^{-2} + \dots + t^{-N} - N t^{-1})].
  \end{split}
\end{equation*}
Then we substitute
\begin{equation*}
  \frac1{1-q} = \frac{1-q^\beta}{1-q}\frac1{1-t^{-1}}
\end{equation*}
to get
\begin{equation*}
  \Omega[\frac{1-q^\beta}{1-q}
  \sum_{i=1}^{N-1} i t^{-(N-i)}]
  = \prod_{m=0}^{\beta-1} \prod_{i=1}^{N-1} (1 - t^{-(N-i)}q^m)^i.
\end{equation*}
\end{NB2}%
Thus the remaining term is
\begin{equation*}
  \begin{split}
  & \exp\left[\sum_{k=1}^\infty \frac{t^{-Nk}}{k}\frac
    {(-N t^{k(N-1)}+(N-1)t^{kN})}
    {(1-q^k)(1-t^k)}\right]
  =   \exp\left[\sum_{k=1}^\infty \frac{1}{k}\frac
    {(-N t^{-k}+(N-1))}
    {(1-q^k)(1-t^k)}\right]
\\
  =\; &
  \Omega[\frac{-Nt^{-1} + N-1}{(1-q)(1-t)}]
  = \Omega[-\frac{-Nt^{-2} + (N-1)t^{-1}}{(1-q)(1-t^{-1})}].
  \end{split}
\end{equation*}
This is well-defined as a formal power series in $q$, $t^{-1}$.
\end{NB}

\section{Refined Hopf link invariants}
\begin{NB}
\begin{equation*}
  \bt = \sqrt{t}, \bq = -\sqrt{q/t}, \ba^2 = t^N \sqrt{t/q}
\end{equation*}
\begin{equation*}
  t = \bt^2, q = \bq^2 \bt^2, t^N = - \ba^2 \bq, 
\end{equation*}
\end{NB}

\subsection{$S$ and $T$ matrices}

We fix $N$.
Let us define the normalized $S$-matrix by
\begin{equation}\label{eq:Snorm}
  \Snorm_{\lambda\mu}(q,t) \defeq
  q^{-\nicefrac{|\lambda||\mu|}N}
  P_{\lambda}(t^{\nicefrac{(N-1)}2},\dots,t^{\nicefrac{(1-N)}2};q,t)
  P_{\mu}(t^{\nicefrac{(N-1)}2} q^{\lambda_1},\dots,
  t^{\nicefrac{(1-N)}2} q^{\lambda_N};q,t),
\end{equation}
as in \cite[(5.12)]{AS}.
Here $\lambda$ and $\mu$ are partitions. We do not impose any
restriction at this moment, but will do soon below.
\begin{NB}
I made a mistake. It was first set
\begin{equation*}
  \Snorm_{\lambda\mu}(q,t) \defeq
  P_{\lambda}(t^{\nicefrac12},\dots,t^{N-\nicefrac12};q,t)
  P_{\mu}(t^{\nicefrac12} q^{-\lambda_1},\dots,t^{N-\nicefrac12} q^{-\lambda_N};q,t).
\end{equation*}
So we need to multiply $t^{1-N/2}$ to the following formulas.
\end{NB}%
\begin{NB}
  The factor $q^{-\nicefrac{|\lambda||\mu|}N}$ is added in July 13. It
  comes as
  \begin{equation*}
    q^{-\nicefrac{|\lambda||\mu|}N}
    P_{\mu}(t^{\nicefrac{(N-1)}2} q^{\lambda_1},\dots,
  t^{\nicefrac{(1-N)}2} q^{\lambda_N};q,t)
  = P_\mu(t^{\nicefrac{(N-1)}2} q^{\lambda_1-\nicefrac{|\lambda|}N},\dots,
  t^{\nicefrac{(1-N)}2} q^{\lambda_N-\nicefrac{|\lambda|}N};q,t)
  \end{equation*}
\end{NB}%

It is known that $\Snorm_{\lambda\mu}(q,t)$ is symmetric under the
exchange of $\lambda$ and $\mu$. This is the Koornwinder-Macdonald
reciprocity formula \cite[VI(6.6)]{Mac}. It is compatible with the
expectation that $\Snorm_{\lambda\mu}$ is the refined colored HOMFLY
polynomial of a Hopf link, where colors are $\lambda$ and $\mu$.
We will see this later in \thmref{thm:symmetry}.

Let us note that this vanishes unless $l(\lambda), l(\mu) \le N$ by
\cite[VI(4.10)]{Mac}.
Note also that \lemref{lem:cut} gives
\begin{equation}\label{eq:trunc}
  \Snorm_{\lambda\mu}(q,t) = 
  \Snorm_{\lambda\overline{\mu}}(q,t),
  \quad
  \Snorm_{\lambda\mu}(q,t) = 
  \Snorm_{\overline{\lambda}\mu}(q,t),
\end{equation}
\begin{NB}
  \begin{equation*}
    q^{-\nicefrac{|\lambda||\mu|}N}
    = q^{-\nicefrac{|\lambda|(|\mu-(1,\dots,1)|+N)}N}
    = q^{-\nicefrac{|\lambda||\mu-(1,\dots,1)|}N} q^{-|\lambda|}.
  \end{equation*}
\end{NB}%
\begin{NB}
This is the version before July 13:
\begin{equation}\label{eq:trunc-orig}
  \Snorm_{\lambda\mu}(q,t) = q^{-|\lambda|(\mu_1-\overline{\mu}_1)}
  \Snorm_{\lambda\overline{\mu}}(q,t),
  \quad
  \Snorm_{\lambda\mu}(q,t) = q^{-|\mu|(\lambda_1-\overline{\lambda}_1)}
  \Snorm_{\overline{\lambda}\mu}(q,t),
\end{equation}
\end{NB}%
where $\overline{\lambda}$, $\overline{\mu}$ are as in
Figure~\ref{fig:2Young} as before.

\begin{NB}
  Let us fix $\overline{\lambda}$, $\overline{\mu}$ and take a sum
  over all $\lambda$, $\mu$ with given $\overline{\lambda}$,
  $\overline{\mu}$. We first consider the sum over $\mu$ only:
  \begin{equation*}
    \sum_{\mu} \Snorm_{\overline\lambda\mu}(q,t)
    = \Snorm_{\overline\lambda\overline\mu}(q,t)\times
    (1 + q^{-|\overline\lambda|} + q^{-2|\overline\lambda|}+\dots)
    = \frac1{1-q^{-|\overline\lambda|}} \Snorm_{\overline\lambda\overline\mu}(q,t),
  \end{equation*}
and more generally
  \begin{equation*}
    \sum_{\mu} \Snorm_{\overline\lambda+(k,\dots,k) \mu}(q,t)
    = \sum_\mu q^{-|\mu|k} \Snorm_{\overline\lambda \mu}(q,t)
    = q^{-|\overline\mu|k} \Snorm_{\overline\lambda\overline\mu}(q,t) \times
    \sum_{l=0}^\infty q^{-Nlk-|\overline\lambda| l}
    = \frac{q^{-|\overline\mu|k}}{1-q^{-|\overline\lambda|-Nk}}
    \Snorm_{\overline\lambda\overline\mu}(q,t).
  \end{equation*}
Therefore
\begin{equation*}
    \sum_{\lambda,\mu} \Snorm_{\lambda\mu}(q,t)
    = \Snorm_{\overline\lambda\overline\mu}(q,t) \times
    \sum_{k,l=0}^\infty q^{-|\overline\mu|k-|\overline\lambda|l-N kl}.
\end{equation*}

Probably this computation is not relevant, as we only take the sum
over one partition in the matrix computation.

\begin{NB2}
On July 13: Now it becomes better:
\begin{equation*}
  \sum_\mu \Snorm_{\overline\lambda\mu}(q,t)
  = (1 + 1 + 1 + \cdots)\Snorm_{\overline\lambda\overline\mu}(q,t).
\end{equation*}
So we get always $\infty$, which is the same, independent of
$\overline\mu$.
\end{NB2}
\end{NB}%

Let us rewrite $\Snorm$ in terms of a modified Macdonald polynomial.
The first term in \eqref{eq:Snorm} is given in \cite[(6.11')]{Mac}:
\begin{equation*}
  P_{\lambda}(t^{\nicefrac{(N-1)}2},\dots,t^{\nicefrac{(1-N)}2};q,t)
   = t^{-\nicefrac{(N-1)|\lambda|}2+n(\lambda)}
   \prod_{s\in\lambda} \frac{1 - q^{a'(s)}t^{N-l'(s)}}
   {1 - q^{a(s)} t^{l(s)+1}}.
\end{equation*}

For example, the case $\lambda=\square$ gives
\begin{equation*}
  P_{\square}(t^{\nicefrac{(N-1)}2},\dots,t^{\nicefrac{(1-N)}2};q,t)
  = t^{\nicefrac{N-1}2} + \dots + t^{\nicefrac{(1-N)}2}
  = \frac{t^{\nicefrac{N}2} - t^{-\nicefrac{N}2}}
  {t^{\nicefrac12}-t^{-\nicefrac12}}.
\end{equation*}
This is $\Snorm_{\square\emptyset}(q,t)$, and is the refined
Chern-Simons partition function for the unknot colored by the vector
representation.

\begin{NB}
  On the other hand, the right hand side gives
  \begin{equation*}
    t^{-(N-1)/2} \frac{1- t^N}{1-t} = \frac{t^{\nicefrac{N}2} - t^{-\nicefrac{N}2}}
    {t^{\nicefrac12}-t^{-\nicefrac12}}.
  \end{equation*}
\end{NB}%
The product of denominator is $c_\lambda(q,t)$ in
\secref{sec:part}. The numerator gives $\Omega[t^N
B_\lambda(q,t^{-1})]^{-1}$ by \eqref{eq:OmegaB}.
Therefore
\begin{equation}\label{eq:first}
  P_{\lambda}(t^{\nicefrac{(N-1)}2},\dots,t^{\nicefrac{(1-N)}2};q,t)
= \frac{t^{-\nicefrac{(N-1)|\lambda|}2 + n(\lambda)}}
  {\Omega[t^{N}B_\lambda(q,t^{-1})]c_\lambda(q,t)}.
\end{equation}
\begin{NB}
  \begin{equation*}
    \begin{split}
    & \Omega[t^{N}B_\lambda(q,t^{-1})]^{-1}
    = \prod_{s\in\lambda} (1 - t^N q^{a'(s)} t^{-l'(s)})
    = (-t)^{N|\lambda|} q^{n(\lambda^t)} t^{-n(\lambda)}
    \prod_{s\in\lambda} (1 - t^{-N} q^{-a'(s)} t^{l'(s)})
\\
   = \; & (-t)^{N|\lambda|} q^{n(\lambda^t)} t^{-n(\lambda)}
    \Omega[t^{-N}B_\lambda(q^{-1},t)]^{-1}.
    \end{split}
  \end{equation*}
\end{NB}%

Next we rewrite the second term using the plethystic substitution
\subsecref{subsec:pleth}. Note that
\begin{equation*}
  \begin{split}
  & t^{\nicefrac{(N-1)}2} q^{\lambda_1} + \dots +
    t^{\nicefrac{(1-N)}2} q^{\lambda_N}
  \begin{NB}
  = t^{\nicefrac{(N-1)}2} \sum_{i=1}^{N} t^{1-i} q^{\lambda_i}    
  \end{NB}%
\\
  =\; & t^{\nicefrac{-(N-1)}2}\left[
  \frac{1 - t^N}{1-t}
  - (1 - q) t^{N-1} \sum_{i=1}^N \sum_{j=1}^{\lambda_i} t^{1-i} q^{j-1}\right]
\\
  =\; & 
  \begin{NB}
  t^{\nicefrac{-(N-1)}2}\left[
  \frac{1 - t^N}{1-t}
  + \frac{t^N (1 - q)(1 - t^{-1}) B_\lambda(q,t^{-1})}{1-t}\right]
  = 
  \end{NB}%
    \frac{t^{\nicefrac{-(N-1)}2}}{1-t}\left(
    1 - t^{N} A_\lambda(q,t^{-1})\right).
  \end{split}
\end{equation*}
Therefore
\begin{equation*}
  P_\mu(
  t^{\nicefrac{(N-1)}2} q^{\lambda_1}, \dots, t^{\nicefrac{(1-N)}2} q^{\lambda_N};
  q,t)
  \begin{NB}
  =
  P_\mu[\frac{t^{\nicefrac{-(N-1)}2}}{1-t}\left(
    1 - t^{N} A_\lambda(q,t^{-1})\right);q,t]
  \end{NB}%
  = t^{\nicefrac{-(N-1)|\mu|}2} P_\mu[\frac{1 - t^{N} A_\lambda(q,t^{-1})}
    {1-t};q,t].
\end{equation*}

\begin{NB}
  If $\lambda=\emptyset$ ($A_\lambda(q,t) = 1$ and $\mu = \square$,
  this is equal to
\(
   t^{-(N-1)/2} (1-t^N)/(1-t) = (t^{N/2} - t^{-N/2})/(t^{1/2} - t^{-1/2}).
\)
It is important that one is power $t^N$ and the other is $t^{-N/2}$ to
have a cancellation.
\end{NB}

Multiplying two factors, we get
\begin{Proposition}\label{prop:Snorm}
\begin{equation*}
  \Snorm_{\lambda\mu}(q,t) 
  = q^{-\nicefrac{|\lambda||\mu|}N}
  \frac{t^{n(\lambda)+n(\mu) - \nicefrac{(N-1)(|\lambda|+|\mu|)}2}}
  {c_\lambda(q,t) c_\mu(q,t)}
  \frac{\tilde H_\mu[1 - t^{N}A_\lambda(q,t^{-1});q,t^{-1}]}
    {\Omega[t^{N}B_\lambda(q,t^{-1})]}.
\end{equation*}
\end{Proposition}

At this stage it is clear that this is well-defined element in
$\Q(q,t)$, as $\tilde H_\mu\in \Z[q^{\pm 1}, t^{\pm 1}]\otimes
\Lambda$. Moreover it vanishes unless $l(\lambda),l(\mu)\le N$ as we
remarked before.

Next we consider the specialization at $q^k t^N = 1$. We understand
$q^{1/N}$ in the expression \eqref{eq:Snorm} as $v$ in
\subsecref{subsec:spec2}. However we keep the notation $q^{1/N}$.

We have the following properties at the specialization at $q^k t^N =
1$.

\begin{Lemma}\label{lem:reduction}
  Assume $l(\lambda), l(\mu)\le N$. Let $\overline{\lambda}$,
  $\overline{\mu}$ as in Figure~\ref{fig:2Young}.
  Then $\Snorm_{\lambda\mu}(q,t)$ is
  well-defined at $q^k t^N = 1$. Moreover, it vanishes at $q^k t^N =
  1$ unless $l(\overline{\lambda}^t), l(\overline{\mu}^t) \le k$.
\end{Lemma}

\begin{proof}
  By the symmetry $\lambda\leftrightarrow \mu$, it is enough to study
  $c_\lambda(q,t)\Omega[t^{N}B_\lambda(q,t^{-1})]$.

 Since we have
\(
   \Snorm_{\lambda\mu}(q,t) = 
   \Snorm_{\overline{\lambda}\mu}(q,t)
\)
by \eqref{eq:trunc}, we may assume $l(\lambda) < N$.

By \lemref{lem:vanish}, $\Omega[t^NB_\lambda(q,t^{-1})]^{-1}$ vanishes
unless $l(\lambda^t)\le k$. Its zero may cancel with the zero of
$c_\lambda(q,t)$, but this cannot happen by \lemref{lem:nonvanish}(2),
as we have assumed $\lambda = \overline\lambda$.
Therefore the both assertions follow.
\end{proof}

We give the duality between $q\leftrightarrow t$. It is related to the
level-rank duality for the affine Lie algebra. This will not be used
later.

We consider $\Snorm_{\lambda^t\mu^t}(t^{-1},q^{-1})$ defined by
replacing $k$ and $N$. The specialization at $q^k t^N = 1$ is defined
as $t = v^{-k}\omega_1$, $q = v^N$ as before. We understand $t^{1/k}$
appearing in $\Snorm_{\lambda^t\mu^t}(t^{-1},q^{-1})$ as
$v^{-1}\omega_1'$ where $\omega_1'$ is chosen so that $(\omega_1')^k =
\omega$. (Therefore $(\omega_1')^{Nk/m} = \omega$.)
Furthermore, we need to consider $\sqrt{t}$, $\sqrt{q}$ variables in
$S_{\lambda\mu}(q,t)$, $S_{\lambda^t\mu^t}(t^{-1}, q^{-1})$
respectively. We have $\sqrt{q}^k \sqrt{t}^N = \omega^{m/2} =
-1$.

\begin{Proposition}
  Suppose $q^k t^N = 1$ and consider
  $\Snorm_{\lambda^t\mu^t}(t^{-1},q^{-1})$ as above. Then
  \begin{equation*}
      \Snorm_{\lambda^t\mu^t}(t^{-1},q^{-1})
      = \frac{\Snorm_{\lambda\mu}(q,t)}
      {\langle P_\lambda, P_\lambda\rangle_{q,t}
        \langle P_\mu,P_\mu\rangle_{q,t}}
      \left(\sqrt{\frac{q}t}\right)^{|\lambda|+|\mu|}
      (\omega_1')^{|\lambda||\mu|}.
  \end{equation*}
\end{Proposition}

\begin{proof}
The left hand side is
\begin{equation*}
  \Snorm_{\lambda^t\mu^t}(t^{-1},q^{-1})
  = t^{\nicefrac{|\lambda||\mu|}k}
  \frac{q^{-n(\lambda^t)-n(\mu^t) + (k-1)\nicefrac{(|\lambda^t|+|\mu^t|)}2}}
  {c_{\lambda^t}(t^{-1},q^{-1}) c_{\mu^t}(t^{-1},q^{-1})}
  \frac{\tilde H_{\mu^t}[1 - q^{-k} A_{\lambda^t}(t^{-1},q);t^{-1},q]}
  {\Omega[q^{-k} B_{\lambda^t}(t^{-1},q)]}.
\end{equation*}
By \eqref{eq:transpose} and some identities in \secref{sec:part}, we get
\begin{equation*}
  \frac{\tilde H_{\mu^t}[1 - q^{-k} A_{\lambda^t}(t^{-1},q);t^{-1},q]}
  {\Omega[q^{-k} B_{\lambda^t}(t^{-1},q)]}
  =   \frac{\tilde H_\mu[1 - t^{N}A_\lambda(q,t^{-1});q,t^{-1}]}
    {\Omega[t^{N}B_\lambda(q,t^{-1})]},
\end{equation*}
and
\begin{equation*}
  \frac{q^{-n(\lambda^t)-n(\mu^t) + (k-1)\nicefrac{(|\lambda^t|+|\mu^t|)}2}}
  {c_{\lambda^t}(t^{-1},q^{-1}) c_{\mu^t}(t^{-1},q^{-1})}
  = \frac{q^{-n(\lambda^t)-n(\mu^t) + (k-1)\nicefrac{(|\lambda^t|+|\mu^t|)}2}}
  {c'_\lambda(q^{-1},t^{-1}) c'_\mu(q^{-1},t^{-1})}
  = \frac{(-1)^{|\lambda|+|\mu|}
    t^{n(\lambda)+n(\mu)}
    q^{(k+1)\nicefrac{(|\lambda^t|+|\mu^t|)}2}}
  {c'_\lambda(q,t) c'_\mu(q,t)}.
\end{equation*}
The assertion follows from 
\(
  \langle P_\lambda,P_\mu\rangle_{q,t} = \delta_{\lambda\mu}
  \nicefrac{c_\lambda'(q,t)}{c_\lambda(q,t)}.
\)
\end{proof}

We now introduce the normalized $T$-matrix as
\begin{equation*}
    \Tnorm_{\lambda\mu}(q,t) = \delta_{\lambda,\mu}
    q^{\|\lambda\|^2/2-\nicefrac{|\lambda|^2}{2N}} t^{-\|\lambda^t\|^2/2 + N|\lambda|/2}.
\end{equation*}
See \cite[(5.13)]{AS}.
Note that we have
\begin{equation}\label{eq:trunc2}
  \Tnorm_{\lambda\mu}(q,t) = \Tnorm_{\lambda\overline{\mu}}(q,t),\quad
  \Tnorm_{\lambda\mu}(q,t) = \Tnorm_{\overline{\lambda}{\mu}}(q,t).
\end{equation}
\begin{NB}
We have
\begin{equation*}
  \begin{split}
  \|{\lambda}^t\|^2 &= 2n({\lambda}) + |{\lambda}|
  = 2n(\overline{\lambda}) + |\overline{\lambda}|
  + N(|\lambda|-|\overline{\lambda}|)
  =   \|\overline{\lambda}^t\|^2   + N(|\lambda|-|\overline{\lambda}|),
\\
  \|{\lambda}\|^2 &= 2n({\lambda}^t) + |{\lambda}|
  = 2n(\overline{\lambda}^t) 
  + 2|\overline{\lambda}|(\lambda_1-\overline{\lambda}_1)
  + N(\lambda_1-\overline{\lambda}_1)(\lambda_1-\overline{\lambda}_1-1)
  + |\overline{\lambda}| + N(\lambda_1-\overline{\lambda}_1)
\\
  & = \|\overline{\lambda}\|^2
  + \frac1N\{|\overline{\lambda}| + N(\lambda_1 - \overline{\lambda}_1)\}^2
  - \frac1N|\overline{\lambda}|^2
  = \|\overline{\lambda}\|^2
  + \frac1N|{\lambda}|^2
  - \frac1N|\overline{\lambda}|^2.
  \end{split}
\end{equation*}
\end{NB}

In \cite[\S5]{AS}, $\Snorm_{\lambda\mu}(q,t)$ appears as
\begin{equation*}
  \Snorm_{\lambda\mu}(q,t) = \langle P_\lambda(x_1,\dots,x_N;q,t)
  |\frac{S}{S_{\emptyset\emptyset}}| 
  P_\mu(x_1,\dots,x_N;q,t)\rangle''_N,
\end{equation*}
where $S$ is the $S$-matrix, and $S_{\emptyset\emptyset}$ is a certain
explicit function, which can be identified with refined Chern-Simons
partition function for $S^3$.
Then the linear operator corresponding to
$\nicefrac{S}{S_{\emptyset\emptyset}}$ is given by
\begin{equation*}
  \begin{split}
   & S^{\lambda}_{\;\mu}(q,t)
\\
  =\; & \frac{\Snorm_{\lambda\mu}(q,t)}
  {\langle P_\lambda, P_\lambda\rangle''_N}
  = q^{-\nicefrac{|\lambda||\mu|}N}
  \frac{t^{n(\lambda)+n(\mu)-\nicefrac{(N-1)(|\lambda|+|\mu|)}2}}
    {c'_\lambda(q,t) c_\mu(q,t)}
  \frac{\tilde H_\mu[1 - t^{N}A_\lambda(q,t^{-1});q,t^{-1}]}
    {\Omega[t^{N-1}q B_\lambda(q,t^{-1})]}
\\
  =\; & q^{-\nicefrac{|\lambda||\mu|}N}
  \frac{t^{n(\lambda)+n(\mu)-\nicefrac{(N-1)(|\lambda|+|\mu|)}2}}
    {c'_\lambda(q,t) c_\mu(q,t)}
    \frac{\Omega[t^N B_\lambda(q,t^{-1})]}{\Omega[t^{N-1}q B_\lambda(q,t^{-1})]}
  \frac{\tilde H_\lambda[1 - t^{N}A_\mu(q,t^{-1});q,t^{-1}]}
    {\Omega[t^N B_\mu(q,t^{-1})]},
  \end{split}
\end{equation*}
where the last expression follows from $\Snorm_{\lambda\mu}(q,t) =
\Snorm_{\mu\lambda}(q,t)$.
When we compute a multiplication of $\Snorm$, we must use
$S^{\lambda}_{\;\mu}(q,t)$. More precisely we consider
$(S^{\lambda}_{\;\mu})$ as a matrix whose indices are partitions
$\lambda$, $\mu$ with $l(\lambda)$, $l(\mu) \le N-1$ and
$l(\lambda^t)$, $l(\mu^t) \le k$.
The same is true for $\Tnorm$.

\begin{NB}
Then
\begin{equation*}
  \frac{S}{S_{00}} P_\mu(x_1,\dots,x_N;q,t) = \sum_{l(\lambda)\le N}
  \frac{\Snorm_{\lambda\mu}(q,t)}{\langle P_\lambda,P_\lambda\rangle''}
  P_\lambda(x_1,\dots,x_N;q,t)
  = \sum_{l(\lambda)\le N} S^{\lambda}_{\;\mu}(q,t) P_\lambda(x_1,\dots,x_N;q,t).
\end{equation*}
\end{NB}%

\begin{NB}
  We define $\tilde{S}_{\lambda\mu}$ by
  \begin{equation*}
    \tilde{S}_{\lambda\mu}(q,t) = \Snorm_{\lambda\mu}(q,t)
    \frac{\langle P_\lambda,P_\lambda\rangle_{q,t}}
    {\langle P_\lambda,P_\lambda\rangle''_{N}}
    = \frac{t^{n(\lambda)+n(\mu) - \nicefrac{(N-1)(|\lambda|+|\mu|)}2}}
  {c_\lambda(q,t) c_\mu(q,t)}
  \frac{\tilde H_\mu[1 - t^{N}A_\lambda(q,t^{-1});q,t^{-1}]}
    {\Omega[t^{N-1}qB_\lambda(q,t^{-1})]}.
  \end{equation*}
Then we have
\begin{equation*}
  \langle P_\lambda| \frac{S}{S_{\emptyset\emptyset}}| P_\mu\rangle_{q,t}
  = \tilde{S}_{\lambda\mu}.
\end{equation*}
Note that $\tilde{S}_{\lambda\mu}$ is no longer symmetric under
$\lambda\leftrightarrow \mu$.
\end{NB}

We define
\begin{equation}\label{eq:kernelfin}
  Z(x,y;q,t) \defeq
  \sum_{\lambda,\mu} \frac{\Snorm_{\lambda\mu}(q,t)}
  {\langle P_\lambda,P_\lambda\rangle''_N
  \langle P_\mu, P_\mu\rangle''_N}
  P_\lambda(x_1,\dots,x_N;q,t) P_\mu(y_1,\dots,y_N;q,t),
\end{equation}
where the summation is over partitions with $l(\lambda)$, $l(\mu) \le
N-1$. We have
\begin{equation*}
  \langle Z(x,y;q,t), f(y)\rangle''_N = 
  \left(\frac{S}{S_{\emptyset\emptyset}} f\right)(x),
\end{equation*}
where the inner product is taken for the variable $y$.  This means
that $Z(x,y;q,t)$ is the kernel function of the operator
$\nicefrac{S}{S_{\emptyset\emptyset}}$.
\begin{NB}
  It is enough to check when $f = P_\mu$. Then it becomes obvious.
\end{NB}%

\begin{Remarks}\label{rem:sum}
(1)
Note that $S^\lambda_{\;\mu}(q,t)$ is always well-defined at $q^k t^N
= 1$ if $l(\lambda)\le N-1$, and we have
$\left.S^\lambda_{\;\mu}(q,t)\right|_{q^k t^N = 1} = 0$ if $l(\mu^t) > k$.
However $\left.S^\lambda_{\;\mu}(q,t)\right|_{q^k t^N = 1}$ may not vanish
even if $l(\lambda^t) > k$.
This does not cause any trouble, as we eventually compute $\langle
\emptyset | \text{a monomial in $S$, $T$,
  $N_{\lambda\bullet}^\bullet$} | \emptyset\rangle$. The answer
remains the same whether imposing $l(\lambda^t), l(\mu^t)\le k$ or
not.
\begin{NB}
  Let $K_{N,k}$ be the linear span of $\{ P_\lambda \mid l(\lambda^t)
  > k\}$. Then the matrix $S|_{q^kt^N=1}$ is zero on $K_{N,k}$.
\end{NB}%
This observation says that the refined Chern-Simons partition function
is a rational function in $q$, $t$ and $t^N$ for a fixed $N$.

Similarly the sum for the kernel function $Z(x,y;q,t)$ in
\eqref{eq:kernelfin} is infinite for generic $q$, $k$, but is a
rational function in $q$, $t$ if we take a part with given degrees in
$x$ and $y$. It gives the operator $\nicefrac{S}{S_{\emptyset\emptyset}}$
at $q^k t^N = 1$.

(2) Let us also consider what happens when we consider $N$ (or $t^N$)
as a variable.
Note that $S^\lambda_{\;\mu}(q,t) = 0$ if $l(\mu) > N$. This is
obvious. For the case $l(\lambda) > N$, we use the second or last
expression for the above $S^\lambda_{\;\mu}(q,t)$ to see that it is
well-defined even if $l(\lambda) > N$, though it may not be $0$. Thus
the situation is similar to (1) above.

However there is a difference: we impose the condition $l(\lambda)$,
$l(\mu)\le N-1$, rather than $\le N$. In fact, when we only assume
$l(\lambda)$, $l(\mu)\le N$, we cannot multiply the matrices
$(S^{\lambda}_{\;\mu})$, $(T^{\lambda}_{\;\mu})$ as infinitely many
entries are nonzero. As entries with $l(\lambda)$ or $l(\mu) = N$ can
be determined by entries with $<N$ thanks to relations
(\ref{eq:trunc}, \ref{eq:trunc2}), the author hope that this is not a
serious issue, even though he does not know a resolution at this moment.
\end{Remarks}

\begin{NB}
For example, let us compute
\begin{equation*}
  \langle P_\lambda(x_1,\dots,x_N;q,t)
  |(\frac{S}{S_{\emptyset\emptyset}})^2| 
  P_\nu(x_1,\dots,x_N;q,t)\rangle''
  = \sum_\mu \frac{\Snorm_{\lambda\mu}(q,t)\Snorm_{\mu\nu}(q,t)}
  {\langle P_\mu,P_\mu\rangle''_N}
\end{equation*}

Recall $\Snorm_{\lambda\mu}(q,t) =
\Snorm_{\mu\lambda}(q,t)$. Therefore this is equal to
\begin{equation*}
  \sum_\mu 
  \frac{\Snorm_{\mu\lambda}(q,t)\Snorm_{\mu\nu}(q,t)}
  {\langle P_\mu,P_\mu\rangle''_N}
  = \sum_\mu
  \begin{aligned}[t]
  & q^{-\nicefrac{(|\lambda|+|\nu|)|\mu|}N}
  \frac{t^{n(\lambda)+2n(\mu)+n(\nu) - \nicefrac{(N-1)(|\lambda|+2|\mu|+|\nu|)}2}}
  {c_\lambda(q,t) c_\mu(q,t) c_\mu'(q,t) c_\nu(q,t)}
  \\
  & \quad\times
  \frac{\tilde H_\lambda[1 - t^{N}A_\mu(q,t^{-1});q,t^{-1}]}
    {\Omega[t^{N-1}q B_\mu(q,t^{-1})]}
  \frac{\tilde H_\nu[1 - t^{N}A_\mu(q,t^{-1});q,t^{-1}]}
    {\Omega[t^{N}B_\mu(q,t^{-1})]}.
  \end{aligned}
\end{equation*}
In order to match with computations in \subsecref{subsec:Hilb}, we
replace $t$ by $t^{-1}$ to have
\begin{equation*}
  \begin{split}
  & \sum_\mu \frac{\Snorm_{\mu\lambda}(q,t^{-1})\Snorm_{\mu\nu}(q,t^{-1})}
  {\left.\langle P_\mu,P_\mu\rangle''_N\right|_{t\leftrightarrow t^{-1}}}
\\
  = \; & \sum_\mu
  q^{-\nicefrac{(|\lambda|+|\nu|)|\mu|}N}
  \frac{t^{-n(\lambda)-2n(\mu)-n(\nu) + \nicefrac{(N-1)(|\lambda|+2|\mu|+|\nu|)}2}}
  {c_\lambda(q,t^{-1}) c_\mu(q,t^{-1}) c_\mu'(q,t^{-1}) c_\nu(q,t^{-1})}
  \frac{\tilde H_\lambda[1 - t^{-N}A_\mu(q,t);q,t]}
    {\Omega[t^{1-N}q B_\mu(q,t)]}
  \frac{\tilde H_\nu[1 - t^{-N}A_\mu(q,t);q,t]}
    {\Omega[t^{-N}B_\mu(q,t)]}.
\\
  =\; &
  \frac{t^{-n(\lambda)-n(\nu) + \nicefrac{(N-1)(|\lambda|+|\nu|)}2}}
  {c_\lambda(q,t^{-1}) c_\nu(q,t^{-1})}
  \sum_\mu
  q^{-\nicefrac{(|\lambda|+|\nu|)|\mu|}N}
  \frac{(-1)^{|\mu|} q^{-n(\mu^t)}t^{-n(\mu)}t^{(N-1)|\mu|}}
    {c'_\mu(q,t^{-1}) c_\mu(q^{-1},t)}
  \frac{\tilde H_\lambda[1 - t^{-N}A_\mu(q,t);q,t]}
    {\Omega[t^{1-N}q B_\mu(q,t)]}
  \frac{\tilde H_\nu[1 - t^{-N}A_\mu(q,t);q,t]}
    {\Omega[t^{-N}B_\mu(q,t)]}.
  \end{split}
\end{equation*}
The last term $\sum_\mu$ can be written in terms of the Hilbert
schemes of points, if we ignore the fact that the partition $\mu$ is
supposed to be contained in the rectangle of size $(N-1)\times k$.

Remark that ${\Omega[t^{1-N}q B_\mu(q,t)]}^{-1}$ and
${\Omega[t^{-N}B_\mu(q,t)]}^{-1}$ are well-defined class even we
substitute $u=t^{-N}$ as they are exterior powers of {\it vector
  bundles}. On the other hand, $t^{N|\mu|}$ needs care, as the sum may diverge.

Let $\overline\mu\subset (N-1)\times k$ denotes the condition that the
corresponding Young diagram is contained in the rectangle of size
$(N-1)\times k$. Then \eqref{eq:trunc} implies that we formally have
\begin{equation*}
  \begin{split}
    & \sum_\mu
    \frac{\Snorm_{\mu\lambda}(q,t^{-1})\Snorm_{\mu\nu}(q,t^{-1})}
    {\left.\langle P_\mu,P_\mu\rangle''_N\right|_{t\leftrightarrow
        t^{-1}}}
    \\
    =\; &
    (1 + 1 + 1 + \cdots)
    \sum_{\overline\mu\subset (N-1)\times k}
    \frac{\Snorm_{\overline\mu\lambda}(q,t^{-1})\Snorm_{\overline\mu\nu}(q,t^{-1})}
    {\left.\langle
        P_{\overline\mu},P_{\overline\mu}\rangle''_N\right|_{t\leftrightarrow
        t^{-1}}}.
  \end{split}
\end{equation*}
It is because \lemref{lem:reduction} implies that Young diagram
$\mu$ with $\Snorm_{\mu\lambda}(q,t)\neq 0$ are those come from
with $\overline{\mu}\subset (N-1)\times k$.
This is a good sign, but how we treat $1 + 1 + \cdots$ ?

We want to have a modified definition of $\Snorm_{\mu\lambda}(q,t)$,
say ${}'\Snorm_{\mu\lambda}(q,t)$, which is
\begin{equation*}
  {}'\Snorm_{\mu\lambda}(q,t)
  = (1 + 1 + 1 + \cdots)^{-1/2} \Snorm_{\mu\lambda}(q,t)
\end{equation*}
so that
\begin{equation*}
  \sum_\mu
    \frac{{}'\Snorm_{\mu\lambda}(q,t^{-1}){}'\Snorm_{\mu\nu}(q,t^{-1})}
    {\left.\langle P_\mu,P_\mu\rangle''_N\right|_{t\leftrightarrow
        t^{-1}}}
    = \sum_{\overline\mu\subset (N-1)\times k}
    \frac{\Snorm_{\overline\mu\lambda}(q,t^{-1})\Snorm_{\overline\mu\nu}(q,t^{-1})}
    {\left.\langle
        P_{\overline\mu},P_{\overline\mu}\rangle''_N\right|_{t\leftrightarrow
        t^{-1}}}.
\end{equation*}
Note also that this already happen at fixing $N$, before imposing
$q^kt^N = 1$.
\end{NB}%

\begin{NB}
We introduce a new variable $\bt$ and consider
\begin{equation*}
 \bt^{\nicefrac{|\lambda|+|\mu|}2}S^{\lambda}_{\;\mu}(q,t).
\end{equation*}
Then
\begin{equation*}
  \bt^{\nicefrac{|\lambda|+|\mu|}2}S^{\lambda}_{\;\mu}(q,t)  
  = \bt^{\nicefrac{|\lambda|-|\overline{\lambda}|}2}
  \bt^{\nicefrac{|\overline{\lambda}|+|\mu|}2}
  S^{\overline{\lambda}}_{\;\mu}(q,t).
\end{equation*}
Therefore
\begin{equation*}
  \sum_{l(\mu)\le N}
  \bt^{\nicefrac{|\lambda|+|\mu|}2} S^{\lambda}_{\;\mu}(q,t)
  \bt^{\nicefrac{|\mu|+|\nu|}2}S^{\mu}_{\;\nu}(q,t)
  = \frac1{1-\bt^N}
  \sum_{l(\overline{\mu}) < N}
  \bt^{\nicefrac{|\lambda|+|\overline\mu|}2}
  S^{\lambda}_{\;\overline\mu}(q,t) 
  \bt^{\nicefrac{|\overline\mu|+|\nu|}2}
  S^{\overline\mu}_{\;\nu}(q,t).
\end{equation*}
However I do not know how to deal with $\bt^N$.
\end{NB}

\subsection{A stable version}

Replacing $t^N$ by $u$, we define a stable version of $\Snorm$:
\begin{equation*}
  \Snorm_{\lambda\mu}(q,t,u,v) \defeq
  v^{-|\lambda||\mu|}
  \frac{t^{\nicefrac{(\|\lambda\|^2 + \|\mu\|^2)}2} u^{-\nicefrac{(|\lambda|+|\mu|)}2}}
  {c_\lambda(q,t) c_\mu(q,t)}
  \frac{\tilde H_\mu[1 - uA_\lambda(q,t^{-1});q,t^{-1}]}
    {\Omega[uB_\lambda(q,t^{-1})]}.
\end{equation*}
Another variable $v$ will be identified with $q^{1/N}$ at the
specialization $q^kt^N = 1$ as in \subsecref{subsec:spec2}, but is an
independent variable at this moment. It is expected that $q^{1/N}$
appears only as the framing factor (see calculation in \cite[\S7]{AS}).

For example, we have
\begin{equation*}
  \Snorm_{\square\emptyset}(q,t,u,v)
  \begin{NB}
    = \frac{t^{1/2} u^{-1/2}}{1 - t} (1 - u)
  \end{NB}%
  = \frac{u^{1/2} - u^{-1/2}}{t^{1/2} - t^{-1/2}}.
\end{equation*}

\begin{NB}
I am not sure whether my understanding of the factor
$q^{-\nicefrac{|\lambda||\mu|}N}$ is correct at this moment.
\end{NB}%

\begin{NB}
Note
\(
   B_\mu(q,t) = (1 - A_\mu(q,t))/(1-t)(1-q).
\)
Therefore
\begin{equation*}
  \Omega[-u B_\mu(q,t)]
  = \Omega\left[
    \frac{-u + uA_\mu(q,t)}{(1-q)(1-t)}
  \right]
  = \Omega\left[
    \frac{-u}{(1-q)(1-t)}
  \right]
  \Omega\left[
   \frac{uA_\mu(q,t)}{(1-q)(1-t)}
   \right].
  \end{equation*}
  \begin{NB2}
    I need to understand where this equality holds.

    Since $\calE/[\C^2] = \frac1{(1-q)(1-t)} \shfO_{X^{[n]}} -
    \shfO^{[n]}$, we have
    \begin{equation*}
      \Wedge_{-u} \shfO^{[n]}
      = \Wedge_{-u} (\frac1{(1-q)(1-t)} \shfO_{X^{[n]}})
      \otimes S_u (\calE/[\C^2]).
    \end{equation*}
  \end{NB2}%
  Since $A_\mu/(1-q)(1-t)$ corresponds to the universal sheaf
  $\calE/[\C^2]$, it produces only a {\it polynomial\/} in $u$. So the
  problem for sum over all partitions $\lambda$ comes from the factor
  $\Omega\left[ \frac{-u}{(1-t)(1-q)} \right]$. 

  However the roles of $u$ in (a) terms $\Omega\left[
    \frac{-u}{(1-q)(1-t)} \right]$ and $\Omega\left[
    \frac{uA_\mu(q,t)}{(1-q)(1-t)} \right]$ and (b)
  $u^{-\nicefrac{(|\lambda|+|\mu|)}2}$ are very different. In (a)
  expressions are formal power series in $u$. On the other hand, if we
  sum over $\lambda$ or $\mu$ in (b), we get formal power series in
  $u^{-1}$. So we must be careful when we multiply them.

  Note also that we lose vanishing properties in \lemref{lem:vanish} in 
  $\Omega\left[
    \frac{uA_\mu(q,t)}{(1-q)(1-t)} \right]$
\end{NB}

\begin{NB}
  A `stable version' of $S^\lambda_\mu$ is
  \begin{equation*}
    Z^\lambda_\mu(q,t)
    = \frac{t^{n(\lambda)+n(\mu)}(tu^{-1})^{\nicefrac{(|\lambda|+|\mu|)}2}}
    {c'_\lambda(q,t) c_\mu(q,t)}
    \frac{\Omega[u B_\lambda(q,t^{-1})]}
    {\Omega[\nicefrac{u q}t B_\lambda(q,t^{-1})]}
  \frac{\tilde H_\lambda[1 - uA_\mu(q,t^{-1});q,t^{-1}]}
    {\Omega[u B_\mu(q,t^{-1})]}.
  \end{equation*}
Or, equivalently we use a `stable version' of the inner product
\begin{equation*}
  \langle P_\lambda, P_\lambda\rangle''
  = \langle P_\lambda, P_\lambda\rangle_{q,t}
  \frac{\Omega[\nicefrac{u q}tB_\lambda(q,t^{-1})]}
  {\Omega[u B_\lambda(q,t^{-1})]}.
\end{equation*}
When we multiply $Z$, the $\Omega[u B_\lambda(q,t^{-1})]$ in the
numerator cancels with the next $\Omega[u B_\lambda(q,t^{-1})]$ in the
denominator. Therefore the multiplication is possible in $\Q(q,t)[u,
u^{-1}]]$.
\end{NB}

We consider a stable version of the kernel function in
\eqref{eq:kernelfin}:
\begin{equation*}
  \begin{NB}
  Z(x,y;q,t,u,v) \defeq
  \end{NB}
  \sum_{\lambda,\mu}
  \frac{\Snorm_{\lambda\mu}(q,t,u,v)}{\langle P_\lambda,P_\lambda\rangle_{q,t}
  \langle P_\mu, P_\mu\rangle_{q,t}}
  P_\lambda(x) P_\mu(y).
\end{equation*}
We do not impose any constraint for the range of $\lambda$, $\mu$. So
this should be considered as an element in the completion with respect
to degrees in $x$, $y$.

Note that this does not give the finite kernel function
\eqref{eq:kernelfin} at $u=t^N$, $q^k t^N = 1$. First because the
inner product is different. Second because the summation is {\it
  not\/} over $l(\lambda)$, $l(\mu)\le N-1$. It is over $\le N$.

The first difference can be resolved by modifying the kernel function
as
\begin{equation}\label{eq:kernelmod}
  \sum_{\lambda,\mu}
  \frac{\Snorm_{\lambda\mu}(q,t,u,v)}{\langle P_\lambda,P_\lambda\rangle_{q,t}
  \langle P_\mu, P_\mu\rangle_{q,t}}
  \frac{\Omega[u B_\lambda(q,t^{-1})]}
  {\Omega[\nicefrac{u q}tB_\lambda(q,t^{-1})]}
  \frac{\Omega[u B_\mu(q,t^{-1})]}
  {\Omega[\nicefrac{u q}tB_\mu(q,t^{-1})]}
  P_\lambda(x) P_\mu(y).
\end{equation}
But the author does not know how to solve the second difference as
mentioned in \remsref{rem:sum}(2).

We change the inner product from $\langle\ ,\ \rangle_{q,t}$ to
$\langle\ ,\rangle_*$ by \eqref{eq:rel} to get
\begin{equation}\label{eq:kernel}
  \begin{split}
    & \left\langle f(x) \middle|
      \sum_{\lambda,\mu}
  \frac{\Snorm_{\lambda\mu}(q,t,u,v)}{\langle P_\lambda,P_\lambda\rangle_{q,t}
  \langle P_\mu, P_\mu\rangle_{q,t}}
  P_\lambda(x) P_\mu(y)
      \middle| g(y)
      \right\rangle_{q,t}
\\
=\; &
  (-t)^{-\deg f-\deg g}
  \left\langle f\left.\left.[\frac{X}{1-t^{-1}}] \middle| 
\sum_{\lambda,\mu}
  \frac{\Snorm_{\lambda\mu}(q,t,u,v)}{\langle P_\lambda,P_\lambda\rangle_{q,t}
  \langle P_\mu, P_\mu\rangle_{q,t}}
  P_\lambda[\frac{X}{1-t}] P_\mu[\frac{Y}{1-t}]        
  \middle| g[\frac{Y}{1-t^{-1}}\right.\right.]\right\rangle_{*}
  \end{split}
\end{equation}
\begin{NB}
\begin{equation*}
  \langle f\left.\left.[\frac{X}{1-t}] \middle| 
      \omega^X \omega^Y Z[\frac{X}{1-t},\frac{Y}{1-t};q,t,u,v] 
  \middle| g[\frac{X}{1-t}\right.\right.]\rangle_{x,y,*,*}
\end{equation*}
\end{NB}%
by \eqref{eq:rel}. Here the inner products are taken for both $x$ and
$y$ simultaneously.
We have
\begin{equation*}
  \begin{split}
   \begin{NB}
     Z[\frac{X}{1-t},\frac{Y}{1-t};q,t,u,v] =
   \end{NB}
   &   \sum_{\lambda,\mu}
  \frac{\Snorm_{\lambda\mu}(q,t,u,v)}{\langle P_\lambda,P_\lambda\rangle_{q,t}
  \langle P_\mu, P_\mu\rangle_{q,t}}
  P_\lambda[\frac{X}{1-t};q,t] P_\mu[\frac{Y}{1-t};q,t]
\\
  =&    \sum_{\lambda,\mu}
  \frac{\Snorm_{\lambda\mu}(q,t,u,v) t^{n(\lambda) + n(\mu)}}
  {c_\lambda'(q,t) c_\mu'(q,t)}
  \tilde H_\lambda(x;q,t^{-1}) \tilde H_\mu(y;q,t^{-1}).
  \end{split}
\end{equation*}
\begin{NB}
  We used \eqref{eq:inverse} and
  \begin{equation*}
    \langle P_\lambda, P_\lambda\rangle_{q,t}
    = \frac{c_\lambda'(q,t)}{c_\lambda(q,t)},
  \end{equation*}
  and ignore the difference between $\langle\ ,\ \rangle''$ and
  $\langle\ ,\ \rangle_{q,t}$.
\end{NB}
In order to match with computations in
\subsecref{subsec:Hilb}, we replace $t$ by $t^{-1}$. 
Then the above becomes
\begin{equation*}\label{eq:Zkernel}
  \begin{NB}
      Z[\frac{X}{1-t^{-1}},\frac{Y}{1-t^{-1}};q,t^{-1},u,v] 
    =     
  \end{NB}%
  \sum_{\lambda,\mu}
  \frac{\Snorm_{\lambda\mu}(q,t^{-1},u,v) t^{-n(\lambda) - n(\mu)}}
  {c_\lambda'(q,t^{-1}) c_\mu'(q,t^{-1})}
  \tilde H_\lambda(x;q,t) \tilde H_\mu(y;q,t).
\end{equation*}
\begin{NB}
  \begin{equation*}
    \frac1{c'_\lambda(q^{-1},t)}
    = \frac1{c'_\lambda(q,t^{-1})}
     \prod_{s\in\lambda} \frac{1-q^{a(s)+1} t^{-l(s)}}{1 - q^{-(a(s)+1)} t^{l(s)}}
    = \frac1{c'_\lambda(q,t^{-1})} t^{-n(\lambda)} q^{n(\lambda^t)} (-q)^{|\lambda|}.
  \end{equation*}
  \begin{equation*}
    \frac1{c_\lambda(q^{-1},t)}
    = \frac1{c_\lambda(q,t^{-1})} t^{-n(\lambda)} q^{n(\lambda^t)} (-t)^{|\lambda|}.
  \end{equation*}
\end{NB}
We absorb the term $\left.(-t)^{-\deg f -\deg
    g}\right|_{t\leftrightarrow t^{-1}}$ in \eqref{eq:kernel} into the
above as $(-t)^{|\lambda|+|\mu|}$. We then substitute
\propref{prop:Snorm} to get
\begin{NB}
\begin{equation*}
  \begin{split}
    & (-t)^{\deg_x+\deg_y}
    Z[\frac{X}{1-t^{-1}},\frac{Y}{1-t^{-1}};q,t^{-1},u,v] 
    = \sum_{\lambda,\mu}
  \frac{(-t)^{|\lambda|+|\mu|} 
    t^{-2n(\lambda)-2n(\mu) + \nicefrac{(N-1)(|\lambda|+|\mu|)}2}
  }
    {c_\lambda(q,t^{-1}) c_\lambda'(q,t^{-1}) c_\mu(q,t^{-1}) c_\mu'(q,t^{-1})
    }
  \frac{\tilde H_\mu[1 - t^{-N}A_\lambda(q,t);q,t]}
    {\Omega[t^{-N}B_\lambda(q,t)]}
  \tilde H_\lambda(x;q,t) \tilde H_\mu(y;q,t),
\\
  = \; &
  \sum_{\lambda,\mu}
  \frac{
  (tu)^{-\nicefrac{(|\lambda|+|\mu|)}2}}
    {\langle \tilde H_\lambda,\tilde H_\lambda\rangle_*
      \langle \tilde H_\mu,\tilde H_\mu\rangle_*}
  \frac{\tilde H_\mu[1 - u A_\lambda(q,t);q,t]}
    {\Omega[u B_\lambda(q,t)]}
  \tilde H_\lambda(x;q,t) \tilde H_\mu(y;q,t),
  \end{split}
\end{equation*}
\end{NB}%
\begin{equation}\label{eq:Zkernel2}
    \sum_{\lambda,\mu}
  \frac{v^{-|\lambda||\mu|}(tu)^{-\nicefrac{(|\lambda|+|\mu|)}2}}
    {\langle \tilde H_\lambda,\tilde H_\lambda\rangle_*
      \langle \tilde H_\mu,\tilde H_\mu\rangle_*}
  \frac{\tilde H_\mu[1 - u A_\lambda(q,t);q,t]}
    {\Omega[u B_\lambda(q,t)]}
 \tilde H_\lambda(x;q,t) \tilde H_\mu(y;q,t),
\end{equation}
where we have used (\ref{eq:temp}, \ref{eq:Hinner}).

\begin{NB}
*************************************************************************

Here is the record of an older attempt.

We have
\begin{equation*}
  Z(x,y) = \sum_\lambda \frac
  {P_{\lambda}(t^{\nicefrac12},\dots,t^{N-\nicefrac12};q,t)}
  {\langle P_\lambda,P_\lambda\rangle''}
  P_\lambda(x;q,t)
  \Omega[\frac{1-t}{1-q} 
    Y (t^{\nicefrac12}q^{-\lambda_1}+\dots+t^{N-\nicefrac12} q^{-\lambda_N})
   ].
\end{equation*}
Since
\begin{equation*}
  \langle f, g\rangle_{q,t} = 
  \langle f, g[\frac{1-q}{1-t}X]\rangle
  = \langle f, \omega g[\frac1{(1-t)^2}X]\rangle_*
  = \langle f[\frac{X}{1-t}], \omega g[\frac{X}{1-t}]\rangle_*,
\end{equation*}
we have
\begin{equation*}
  \begin{split}
  &\frac{S}{S_{00}} H_\mu(x;q,t)
  = \langle Z(x,y), H_\mu(y;q,t)\rangle''_y
  = \langle Z[x,\frac{Y}{1-t}], \omega H_\mu[\frac{Y}{1-t};q,t]\rangle_{*,y}
  = \langle Z[x,\frac{Y}{(1-t)^2}], \omega H_\mu(y;q,t)\rangle_{*,y}
\\
  =\; &
  \sum_\lambda
  \frac
  {P_{\lambda}(t^{\nicefrac12},\dots,t^{N-\nicefrac12};q,t)}
  {\langle P_\lambda,P_\lambda\rangle''}
  P_\lambda(x;q,t)
  \langle \Omega[\frac1{(1-q)(1-t)} Y
  (t^{\nicefrac12}q^{-\lambda_1}+\dots+t^{N-\nicefrac12} q^{-\lambda_N})]
  , \omega H_\mu(y;q,t)\rangle_{*,y}.
  \end{split}
\end{equation*}
Since $\omega\Omega[XY/(1-q)(1-t)]$ is the reproducing kernel for the inner
product $\langle\ ,\ \rangle_*$, the second term is
\begin{equation*}
  H_\mu[t^{\nicefrac12}q^{-\lambda_1}+\dots+t^{N-\nicefrac12} q^{-\lambda_N};q,t]
  = H_\mu[\frac{t^{N-\nicefrac12}}{1-t^{-1}}(1- t^{-N}A_\lambda(q^{-1},t));q,t].
\end{equation*}
Therefore
\begin{equation*}
  \frac{S}{S_{00}} H_\mu(x;q,t)
  =   \sum_\lambda
  \frac
  {P_{\lambda}(t^{\nicefrac12},\dots,t^{N-\nicefrac12};q,t)}
  {\langle P_\lambda,P_\lambda\rangle''}
  P_\lambda(x;q,t)
  H_\mu[\frac{t^{N-\nicefrac12}}{1-t^{-1}}(1- t^{-N}A_\lambda(q^{-1},t));q,t].
\end{equation*}

It is probably better to consider
\begin{equation*}
  \begin{split}
    & \langle Z(x,y), f(y)\rangle''_y = \langle Z[x,\frac{Y}{1-t}],
    \omega f[\frac{Y}{1-t}]\rangle_*
    \\
    =\; & \sum_{\lambda,\mu} \frac{\Snorm_{\lambda\mu}(q,t)}{\langle
      P_\lambda,P_\lambda\rangle'' \langle P_\mu, P_\mu\rangle''}
    P_\lambda(x;q,t) \langle P_\mu[\frac{Y}{1-t},q,t],
    \omega f[\frac{Y}{1-t}]\rangle_*
\\
    =\; & \sum_{\lambda,\mu} \frac{\Snorm_{\lambda\mu}(q,t)}{\langle
      P_\lambda,P_\lambda\rangle'' \langle P_\mu, P_\mu\rangle'' c_\mu}
    P_\lambda(x;q,t) \langle \tilde H_\mu(y;q,t^{-1}),
    \omega f[\frac{Y}{1-t}]\rangle_*,
  \end{split}
\end{equation*}
and
\begin{equation*}
  \begin{split}
    & \langle g(x), Z(x,y), f(y)\rangle''_{x,y}
  = \langle \omega g[\frac{X}{1-t}] Z[\frac{X}{1-t},\frac{Y}{1-t}]
    \omega f[\frac{Y}{1-t}]\rangle_*
    \\
    =\; & \sum_{\lambda,\mu} \frac{\Snorm_{\lambda\mu}(q,t)}{\langle
      P_\lambda,P_\lambda\rangle'' \langle P_\mu, P_\mu\rangle''}
    P_\lambda(x;q,t) \langle P_\mu[\frac{Y}{1-t},q,t],
    \omega f[\frac{Y}{1-t}]\rangle_*
\\
    =\; & \sum_{\lambda,\mu} \frac{\Snorm_{\lambda\mu}(q,t)}{\langle
      P_\lambda,P_\lambda\rangle'' \langle P_\mu, P_\mu\rangle'' c_\lambda c_\mu}
    \langle \omega g[\frac{X}{1-t}], \tilde H_\lambda(x;q,t^{-1})
    \rangle_*
    \langle \tilde H_\mu(y;q,t^{-1}),
    \omega f[\frac{Y}{1-t}]\rangle_*,
  \end{split}
\end{equation*}

It is probably better to consider
\begin{equation*}
    \frac{S}{S_{00}} H_\mu[(1-t)X;q,t]
    = \sum_\lambda
  \frac
  {P_{\lambda}(t^{\nicefrac12},\dots,t^{N-\nicefrac12};q,t)}
  {\langle P_\lambda,P_\lambda\rangle''}
  P_\lambda(x;q,t)
  H_\mu[-{t^{N+\nicefrac12}}(1- t^{-N}A_\lambda(q^{-1},t));q,t].
\end{equation*}
Then
\begin{equation*}
  \langle H_\lambda[(1-t)X;q,t]
      \frac{S}{S_{00}} H_\mu[(1-t)X;q,t]\rangle_*
  = \sum_\nu
  \frac
  {P_{\nu}(t^{\nicefrac12},\dots,t^{N-\nicefrac12};q,t)}
  {\langle P_\nu,P_\nu\rangle''}
  \langle H_\lambda[(1-t)X;q,t], P_\nu(x;q,t)\rangle_*
  H_\mu[-{t^{N+\nicefrac12}}(1- t^{-N}A_\lambda(q^{-1},t));q,t].      
\end{equation*}
\end{NB}

\begin{NB}
\subsection{A change of the base}

Let us rewrite \eqref{eq:Zkernel2}. Note
\begin{equation*}
  \sum_\mu (-1)^{|\mu|} (t/u)^{|\mu|/2}
  \frac{\tilde H_\mu(z;q,t) \tilde H_\mu(y;q,t)}
  {\langle \tilde H_\mu, \tilde H_\mu\rangle_*}
  = \tilde\Omega[\frac{\epsilon\sqrt{t/u}YZ}{(1-q)(1-t)}]
  = \sum_\mu (-1)^{|\mu|} (t/u)^{|\mu|/2}
  s_\mu[\frac{Z}{1-q}] s_{\mu^t}[\frac{Y}{1-t}].
\end{equation*}
Therefore
\begin{equation*}
   Z[\frac{X}{1-t^{-1}},\frac{Y}{1-t^{-1}};q,t^{-1}] 
   = \sum_{\lambda,\mu} 
    \frac{(-1)^{|\lambda|+|\mu|}
  (t/u)^{\nicefrac{(|\lambda|+|\mu|)}2}}
    {\langle \tilde H_\lambda,\tilde H_\lambda\rangle_*
    \Omega[u B_\lambda(q,t)]}
  s_{\mu}[\frac{1-uA_\lambda(q,t)}{1-q}] 
  \tilde H_\lambda(x;q,t) s_{\mu^t}[\frac{Y}{1-t}].
\end{equation*}
Let
\begin{equation*}
  \tilde s_\mu(x;q,t) \defeq
   \frac1{\Omega[u X]} s_\mu[\frac{1-u(1-(1-q)(1-t)X)}{1-q}].
\end{equation*}
Then 
\begin{equation*}
  \begin{split}
   & Z[\frac{X}{1-t^{-1}},\frac{Y}{1-t^{-1}};q,t^{-1}] 
   = \sum_{\lambda,\mu} 
    \frac{(-1)^{|\lambda|+|\mu|}
  (t/u)^{\nicefrac{(|\lambda|+|\mu|)}2}}
    {\langle \tilde H_\lambda,\tilde H_\lambda\rangle_*}
  \Delta_{\tilde s_{\mu}}
  \tilde H_\lambda(x;q,t) s_{\mu^t}[\frac{Y}{1-t}]
\\
  =\; & \sum_{k,\mu} 
    (-1)^{k+|\mu|}
  (t/u)^{\nicefrac{(k+|\mu|)}2}
  \Delta_{\tilde s_{\mu}} e_{k}[\frac{X}{(1-q)(1-t)}]
  s_{\mu^t}[\frac{Y}{1-t}].
  \end{split}
\end{equation*}
\end{NB}

\section{Geometric interpretations via Hilbert schemes}

We give geometric interpretations of the stable version of the kernel
function \eqref{eq:kernel} or \eqref{eq:Zkernel2} in the previous
section via Hilbert schemes of points on the plane. We give two
versions, a naive one and the second one. The second one looks better
though it contains an infinite sum.

Let us first replace the inner product $\langle\ ,\ \rangle_*$ in
\eqref{eq:kernel} by $(\ ,\ )$ in \eqref{eq:geom}, where we suppress
the functor $\Phi$. It corresponds to applying operators $\nabla^x$,
$\nabla^y$ to $\tilde H_\lambda(x;q,t)$, $\tilde H_\mu(y;q,t)$ in
\eqref{eq:Zkernel2}. Here $\nabla^x$ means the $\nabla$-operator for
the variable $x$ and similarly for $\nabla^y$.
We denote the result by $Z(x,y;q,t,u,v)$, i.e.,
\begin{equation*}
  Z(x,y;q,t,u,v) \defeq
      \sum_{\lambda,\mu}
  \frac{v^{-|\lambda||\mu|}(tu)^{-\nicefrac{(|\lambda|+|\mu|)}2}}
    {\langle \tilde H_\lambda,\tilde H_\lambda\rangle_*
      \langle \tilde H_\mu,\tilde H_\mu\rangle_*}
  \frac{\tilde H_\mu[1 - u A_\lambda(q,t);q,t]}
    {\Omega[u B_\lambda(q,t)]}
 \nabla^x \tilde H_\lambda(x;q,t) \nabla^y \tilde H_\mu(y;q,t).
\end{equation*}
\begin{NB}
  This is different from $Z$ in the previous section.
\end{NB}%
Therefore
\begin{equation}\label{eq:ZZZ}
  \left.\eqref{eq:kernel}\right|_{t\leftrightarrow t^{-1}}
 = \left( f\left.\left.[\frac{X}{1-t}] \middle|
      Z(x,y;q,t,u,v)
  \middle| g[\frac{Y}{1-t}\right.\right.]\right).
\end{equation}

Recall that
\(
   f[\nicefrac{X}{1-t}]
\)
and
\(
   g[\nicefrac{Y}{1-t}]
\)
are considered as classes in $K_{\C^*\times\C^*}(L^{[n]}_x)$ under $\Phi$
\eqref{eq:KL}.
The expression
\(
      Z(x,y;q,t,u,v)
\)
in the middle is a generating function of classes in the localized
equivariant $K$-group $K_{\C^*\times\C^*}(X^{[n]}\times
X^{[n]})\otimes_{\Z[q^{\pm 1},t^{\pm 1}]} \otimes \Q(q,t)$.
As the inner product $(\ ,\ )$ is simply given by the functor $\mathbf
R\Gamma_{X^{[n]}}(-\otimes^{\mathbf L}-)$, the above \eqref{eq:ZZZ}
is simply given by the convolution product
\begin{equation*}
  \mathbf R\Gamma_{X^{[n]}\times X^{[n]}}(
  p_1^*(f[\frac{X}{1-t}])\otimes^{\mathbf L}
      Z(x,y;q,t,u,v)
  \otimes^{\mathbf L}
  p_2^*(g[\frac{Y}{1-t}] )).
\end{equation*}

Recall that we have a modified version of $Z(x,y;q,t,u,v)$ in
\eqref{eq:kernelmod}, where we have an extra factor
\begin{equation*}
  \frac{\Omega[u B_\lambda(q,t)]}{\Omega[uqt B_\lambda(q,t)]}
    \frac{\Omega[u B_\mu(q,t)]}{\Omega[uqt B_\mu(q,t)]}.
\end{equation*}
(Recall that we replaced $t$ by $t^{-1}$.)
This factor appears as
\(
   \Delta_f^x \tilde H_\lambda(x;q,t) \Delta_f^y\tilde H_\mu(y;q,t)
\)
with
\begin{equation*}
   f(x) = \frac{\Omega[uX]}{\Omega[uqtX]}.
\end{equation*}
Under $\Phi$, $\Delta_f$ gives an operator multiplying
\begin{equation*}
  \frac{\Wedge_{-u} qt\shfO^{[n]}}{\Wedge_{-u} \shfO^{[n]}}.
\end{equation*}
We do not understand whether $Z$ or its modified version is correct
yet, but we only treat $Z(x,y;q,t,u,v)$ as the difference is just
given above.

\subsection{A naive geometric interpretation}

First note that
\begin{equation*}
  \frac{\nabla^x \tilde H_\lambda(x;q,t)}
  {\langle \tilde H_\lambda,\tilde H_\lambda\rangle_*}
  = \frac{I_\lambda}{\ch\Wedge_{-1} T^*_{I_\lambda} X^{[|\lambda|]}}
\end{equation*}
by \eqref{eq:HilbTang} and that $\tilde H_\lambda$ corresponds to the
class $I_\lambda$ for the fixed point as we explained in
\subsecref{subsec:Haiman}.

The factor 
\(
  \nicefrac1{\Omega[u B_\lambda(q,t)]}
\)
is the character
\(
   \ch \Wedge_{-u} \shfO^{[n]}_{I_\lambda}
\)
of the exterior power of the tautological bundle at the point
$I_\lambda$ by \eqref{eq:taut}.

As $\tilde H_\mu$ is a linear combination
\(
   \tilde H_\mu = \sum_\nu \tilde K_{\nu\mu}(q,t) s_\nu
\)
of Schur functions, we have
\begin{equation*}
  \tilde H_\mu[1 - u A_\lambda;q,t]
  = \sum_{\nu} \tilde K_{\nu\mu}(q,t)
  s_\nu[1 - u A_\lambda].
\end{equation*}
Since $A_\lambda(q,t)$ is the character
\(
  \ch \calE_{(0,I_\lambda)}
\)
by \eqref{eq:universal}, we can interpret
\(
  s_\nu[1 - u A_\lambda]
\)
as the fiber at $I_\lambda$ of
$s_\nu[\shfO - u \calE/[0]]$,
the Schur functor $s_\nu$ applied to $\shfO - u \calE/[0]$.
Moreover $\tilde K_{\nu\mu}(q,t)$ is given by the Procesi bundle
$\mathcal P$:
\begin{equation*}
  \tilde K_{\nu\mu}(q,t) = \ch \mathcal (P_\nu)_{I_\mu}.
\end{equation*}

Therefore $Z(x,y;q,t,u,v)$ is the sum of the fiber of
\begin{equation*}
  \bigoplus_\nu
  (\Wedge_{-u} \shfO^{[n]}\otimes s_\nu[\shfO - u\calE/[0]])
    \boxtimes \mathcal P_\nu
\end{equation*}
at $(I_\lambda, I_\mu)$, divided by
\begin{equation*}
  {\ch \Wedge_{-1} T^*_{I_\lambda} X^{[|\lambda|]}
    \ch \Wedge_{-1} T^*_{I_\mu} X^{[|\mu|]}}.
\end{equation*}
This is the first geometric interpretation of $Z(x,y;q,t,u,v)$. It is
a compact form, but seems to difficult to its geometric meaning.

\subsection{The second geometric interpretation}\label{subsec:second}

In \cite[The proof of Th.~3.3]{GHT} the symmetry of
$\Snorm_{\lambda\mu}$ is proved by an application of the following:
\begin{Theorem}\label{thm:symmetry}
  \begin{equation*}
    \frac{\tilde H_\mu[1 - u A_\lambda(q,t);q,t]}
  {\Omega[u B_\lambda(q,t)]}
  = \Omega[\frac{-u}{(1-t)(1-q)}]
  \left\langle
    \nabla^{-1} \left(\tilde\Omega[\frac{-X A_\mu(q,t)}{(1-q)(1-t)}]\right), 
    \tilde\Omega[\frac{-uX A_\lambda(q,t)}{(1-q)(1-t)}]
    \right\rangle_*
  \end{equation*}
\end{Theorem}

\begin{NB}
  Recall $u= t^{-N}$ is substituted later. Both $u$ are positive
  powers in the left and right hand side. It seems that this does not
  match (or does match ?) with \propref{prop:Snorm}.
\end{NB}

Since $\nabla$ is self-adjoint with respect to $\langle\ ,\ \rangle_*$
from \eqref{eq:Hinner}, or its geometric interpretation in
\subsecref{subsec:Haiman}, the symmetry under
$\lambda\leftrightarrow\mu$ is clear.

It should be also remarked that this is an equality in
$\Q(q,t)[[u]]$. The left hand side is in $\Z[q^{\pm 1}, t^{\pm 1},
u]$, but the right hand side contains an infinite sum in powers of $u$.
\begin{NB}
  Therefore it is not clear at this moment that why we can put 
  $u$ at $t^{-N}$.
\end{NB}

Note that there is a typo in \cite{GHT}. The second $u$ was missing.

\begin{NB}
By \cite[5.4.7]{Haiman2} we have
\begin{equation*}
    \left\langle
    \nabla^{-1} \tilde\Omega[\frac{-X A_\mu(q,t)}{(1-q)(1-t)}], 
    \tilde\Omega[\frac{-uX A_\lambda(q,t)}{(1-q)(1-t)}]
    \right\rangle_*
    =
    \left\langle
    \tilde\Omega[\frac{-X A_\mu(q,t)}{(1-q)(1-t)}], 
    \tilde\Omega[\frac{-uX A_\lambda(q,t)}{(1-q)(1-t)}]
    \right\rangle_{.},
\end{equation*}
\end{NB}%

Let us understand the right hand side in terms of Hilbert schemes. Let
us first remark that
\begin{equation*}
  \Omega[\frac{-u}{(1-t)(1-q)}] = \Omega[\frac{u}{(1-t)(1-q)}]^{-1}
  = \left(\sum_{n=0}^\infty u^n \ch\Gamma_{S^nX}(\shfO_{S^nX})\right)^{-1},
\end{equation*}
which we calculated as
\begin{equation*}
  \left(
  \sum_{n=0}^\infty u^n \ch {\mathbf R}\Gamma_{X^{[n]}}(\shfO_{X^{[n]}})
  \right)^{-1}
\end{equation*}
in \eqref{eq:S3}.

The second term can be also calculated to get
\begin{Theorem}\label{thm:HilbHilb}
  \begin{equation}\label{eq:HilbHilb}
    \begin{split}
  &\frac{\tilde H_\mu[1 - u A_\lambda(q,t);q,t]}
  {\Omega[u B_\lambda(q,t)]}
  \sum_{n=0}^\infty u^n \ch \mathbf R\Gamma_{X^{[n]}}(\shfO_{X^{[n]}})
\\
=\; &
  \sum_{n=0}^\infty u^n \ch \mathbf R\Gamma_{X^{[n]}}(
  (\omega\nabla^{-1}\tilde H_\lambda)[\uE;q,t]
  \otimes (\omega\nabla^{-1}\tilde H_\mu)[\uE;q,t]).
    \end{split}
\end{equation}
\end{Theorem}

Here $(\omega\nabla^{-1}\tilde H_\lambda)[\uE;q,t]$ is the plethystic
substitution of the universal sheaf $\uE$ at the fiber $0\times
X^{[n]}$ to $(\omega\nabla^{-1}\tilde H_\lambda)(x;q,t)$. A little
more concretely, as
\begin{equation*}
  \omega\nabla^{-1} \tilde H_\lambda(x;q,t)
  = \tilde H_\lambda(x;q^{-1},t^{-1})
  = \sum_{\sigma} \tilde K_{\sigma\lambda}(q^{-1},t^{-1}) s_\sigma,
\end{equation*}
it is nothing but
\begin{equation*}
  (\omega\nabla^{-1}\tilde H_\lambda)[\uE;q,t]
  = \bigoplus_\sigma \tilde K_{\sigma\lambda}(q^{-1},t^{-1})
  s_\sigma[\uE], 
\end{equation*}
where $s_\sigma[\uE]$ is the Schur functor applied to $\uE$, as in the
previous subsection.

\begin{NB}
By the positivity theorem \cite{Haiman}, we have $\tilde
K_{\sigma\lambda}\in \Z_{\ge 0}[q,t]$, and hence
$\omega\nabla^{-1}H_\lambda[\uE]$ makes sense.
\begin{NB2}
  Since the universal sheaf is not a vector bundle, I am not sure that
  I can apply Haiman's vanishing theorem.
\end{NB2}
\end{NB}

Before giving the proof of \thmref{thm:HilbHilb}, let us
consider a special case $\lambda=\emptyset$. We can slightly change the
expression.
\begin{Corollary}\label{cor:empty}
  \begin{equation*}
  (\nabla s_{\lambda^t})[1-u]
  \sum_{n=0}^\infty u^n \ch \mathbf R\Gamma_{X^{[n]}}(\shfO_{X^{[n]}})  
  = 
  \sum_{n=0}^\infty u^n \ch \mathbf R\Gamma_{X^{[n]}}(
  s_\lambda[\uE]).
\end{equation*}
\end{Corollary}

This special case is already very surprising formula as it computes
holomorphic Euler characteristics of universal sheaves twisted by
arbitrary Schur functors. Note that $\nabla$ is defined in terms of
$\tilde H_\mu$, and hence is very complicated to compute $\nabla
s_{\lambda^t}$.

\begin{proof}
\thmref{thm:HilbHilb} with $\lambda=\emptyset$ gives
\begin{equation}\label{eq:qtdim}
  \tilde H_\mu[1-u;q,t] 
  = \Omega[\frac{-u}{(1-t)(1-q)}]
  \sum_{n=0}^\infty u^n \ch \mathbf R\Gamma_{X^{[n]}}(
  (\omega\nabla^{-1} \tilde H_\mu)[\uE;q,t]).
\end{equation}
\begin{NB}
Note that
\begin{equation*}
  \tilde H_\mu[1-u;q,t] 
  = \Omega[u B_\mu(q,t)]^{-1} = \ch\Wedge_{-u}\shfO^{[n]}_{I_\mu}.
\end{equation*}
This may suggest that we should understand the formula something on
$X^{[n]}\times X^{[n]}$.
\end{NB}%

If we write
\begin{equation*}
  \omega \nabla^{-1} \tilde H_\mu(x;q,t) 
  = \sum_\lambda \tilde K_{\lambda\mu}(q^{-1},t^{-1}) s_\lambda, 
\end{equation*}
we have
\begin{equation*}
  \tilde H_\mu(x;q,t) = \sum_\lambda \tilde K_{\lambda\mu}(q^{-1},t^{-1})
  \nabla s_{\lambda^t}.
\end{equation*}
\begin{NB}
  $\omega s_{\lambda} = s_{\lambda^t}$.
\end{NB}
Since the matrix $(\tilde K_{\lambda\mu}(q^{-1},t^{-1}))$ is
invertible, we have the assertion.
\end{proof}

Let us start the proof of \thmref{thm:HilbHilb}. We first prepare the
following.

\begin{Lemma}\label{lem:symmetry}
  \begin{equation*}
    (-1)^{|\lambda|}\tilde H_\mu[-A_\lambda(q,t)] q^{n(\lambda^t)} t^{n(\lambda)}
    = (-1)^{|\mu|}\tilde H_\lambda[-A_\mu(q,t)] q^{n(\mu^t)} t^{n(\mu)}.
  \end{equation*}
\end{Lemma}

\begin{proof}
  Let us consider the left hand side of \eqref{eq:HilbHilb}, where the
  denominator is replaced by the middle term in \eqref{eq:taut}. Since
  it is symmetric under the $\lambda$, $\mu$, we have
\begin{equation*}
  \tilde H_\mu[1 - u A_\lambda(q,t);q,t]
  \prod_{s\in\lambda}(1 - u q^{a'(s)} t^{l'(s)})
  =   \tilde H_\lambda[1 - u A_\mu(q,t);q,t]
  \prod_{s\in\mu}(1 - u q^{a'(s)} t^{l'(s)}).
\end{equation*}
Therefore
\begin{equation*}
  \tilde H_\mu[u^{-1} - A_\lambda(q,t);q,t]
  \prod_{s\in\lambda}(u^{-1} - q^{a'(s)} t^{l'(s)})
  =   \tilde H_\lambda[u^{-1} - A_\mu(q,t);q,t]
  \prod_{s\in\mu}(u^{-1} - q^{a'(s)} t^{l'(s)}).
\end{equation*}
Setting $u^{-1} = 0$, we get the assertion.
\end{proof}

Now we are ready to calculate the second term in the right hand side
of \thmref{thm:symmetry}. By (\ref{eq:Cauchy},\ref{eq:Hinner}) we have
\begin{equation*}
  \begin{split}
    &\left\langle
    \nabla^{-1} \tilde\Omega[\frac{-X A_\mu(q,t)}{(1-q)(1-t)}], 
    \tilde\Omega[\frac{-uX A_\lambda(q,t)}{(1-q)(1-t)}]
    \right\rangle_*
\\
    = \; & \sum_\nu 
    \frac{q^{-n(\nu^t)} t^{-n(\nu)} \tilde H_\nu[- A_\mu(q,t)]}
    {c_\nu(q^{-1},t) c'_\nu(q,t^{-1})}
    \frac{q^{-n(\nu^t)} t^{-n(\nu)} \tilde H_\nu[- uA_\lambda(q,t)]}
    {c_\nu(q^{-1},t) c'_\nu(q,t^{-1})}
    \langle \nabla^{-1} \tilde H_\nu(x;q,t), \tilde H_\nu(x;q,t)\rangle_*
\\
   =\; & \sum_\nu 
    \frac{q^{-2n(\nu^t)} t^{-2n(\nu)} 
     \tilde H_\nu[- A_\mu(q,t)] \tilde H_\nu[- uA_\lambda(q,t)]}
    {c_\nu(q^{-1},t) c'_\nu(q,t^{-1})}.
  \end{split}
\end{equation*}
By \lemref{lem:symmetry} this is equal to
\begin{equation*}
  \sum_\nu  \frac{(-1)^{|\lambda|+|\mu|}
    q^{-n(\lambda^t)-n(\mu^t)} t^{-n(\lambda)-n(\mu)} u^{|\nu|} 
      \tilde H_\lambda[-A_\nu(q,t)] \tilde H_\mu[-A_\nu(q,t)]}
    {c_\nu(q^{-1},t) c'_\nu(q,t^{-1})}.
\end{equation*}
We have
\begin{equation*}
  (-1)^{|\lambda|} q^{-n(\lambda^t)} t^{-n(\lambda)}
  \tilde H_\lambda[-A_\nu(q,t)]
  = (\omega \nabla^{-1} \tilde H_\lambda)[A_\nu(q,t)].
\end{equation*}

Note that $A_\nu(q,t)$ is the character of $\calE|_{0\times I_\nu}$,
and the denominator is $\Wedge_{-1} T^*_{I_\nu} X^{[n]}$ by
\eqref{eq:HilbTang}. Hence the Atiyah-Bott-Lefschetz fixed point
formula gives the right hand side of \thmref{thm:HilbHilb}.
This completes the proof of \thmref{thm:HilbHilb}.

We substitute \thmref{thm:HilbHilb} to the definition of
$Z(x,y;q,t,u,v)$ to get
\begin{equation*}
  \begin{split}
  &   Z(x,y;q,t,u,v)
\\
  = \; & \sum_{\lambda,\mu}
  \begin{aligned}[t]
  & \frac{v^{-|\lambda||\mu|}(tu)^{-\nicefrac{(|\lambda|+|\mu|)}2}}
  {\langle \tilde H_\lambda, \tilde H_\lambda\rangle_*
    \langle \tilde H_\mu, \tilde H_\mu\rangle_*}
 \nabla^x  \tilde H_\lambda(x;q,t) \nabla^y \tilde H_\mu(y;q,t)
\\
  &\quad\times  
  \frac{\sum_{n=0}^\infty u^n \ch\mathbf R\Gamma_{X^{[n]}}(
  (\omega\nabla^{-1}\tilde H_\lambda)[\uE;q,t]\otimes
  (\omega\nabla^{-1}\tilde H_\mu)[\uE;q,t]}
  {\sum_{n=0}^\infty u^n \ch \mathbf R\Gamma_{X^{[n]}}(\shfO_{X^{[n]}})}.
  \end{aligned}
  \end{split}
\end{equation*}
\begin{NB}
  Here $u$ must be replaced by $t^{-N}$, which looks a little strange.
\end{NB}%
By \eqref{eq:Cauchy}
\begin{equation*}
  \begin{split}
  & \sum_\lambda \frac{\nabla^x \tilde H_\lambda(x;q,t) 
    \omega(\nabla^z)^{-1}\tilde H_\lambda(z;q,t)}
  {\langle \tilde H_\lambda, \tilde H_\lambda\rangle_*}
  = \sum_\lambda \frac{\tilde H_\lambda(x;q,t) 
    \omega\tilde H_\lambda(z;q,t)}
  {\langle \tilde H_\lambda, \tilde H_\lambda\rangle_*}
\\
  = \; & 
  \omega^z \tilde \Omega\left[\frac{XZ}{(1-q)(1-t)}\right]
  = \Omega\left[\frac{XZ}{(1-q)(1-t)}\right].
  \end{split}
\end{equation*}
Using the usual Cauchy formula for Schur functions, we write this as
\begin{equation*}
  \sum_\lambda s_\lambda[\frac{X}{1-q}]
  s_{\lambda}[\frac{Z}{1-t}].
\end{equation*}
Since
\(
   \nicefrac1{1-q} = \ch \C[y],
\)
$\uE/(1-q)$ means the slant product $\calE/\{x=0\}$. Similarly
$\uE/(1-t)$ means $\calE/\{y=0\}$.

Therefore we get our main result.
\begin{Theorem}\label{thm:main}
\begin{equation}\label{eq:ZZ}
  Z(x,y;q,t,u,v)
  = \sum_{\lambda,\mu}
  \begin{aligned}[t]
  & v^{-|\lambda||\mu|}(tu)^{-\nicefrac{(|\lambda|+|\mu|)}2}
   s_\lambda[\frac{X}{1-t}] s_\mu[\frac{Y}{1-q}]
\\
  &\quad\times  
  \frac{
    \sum_{n=0}^\infty u^n \ch\mathbf R\Gamma_{X^{[n]}}(
    s_\lambda[\calE/\{ x=0\}]\otimes
    s_\mu[\calE/\{ y=0\}])
  }
  {\sum_{n=0}^\infty u^n \ch \mathbf R\Gamma_{X^{[n]}}(\shfO_{X^{[n]}})}.
  \end{aligned}
\end{equation}
\end{Theorem}

Note that $s_\lambda[X/(1-t)]$, $s_\mu[Y/(1-q)]$ give bases for
$\bigoplus K_{\C^*\times\C^*}(L^{[n]}_y)$, $\bigoplus
K_{\C^*\times\C^*}(L^{[n]}_x)$ respectively by \subsecref{subsec:lag}.
Therefore the above is a class in 
$\bigoplus K_{\C^*\times\C^*}(L^{[m]}_x\times L_y^{[n]})$.

Let us give several further remarks on this expression.

\begin{Remarks}
  (1) The second term in the right hand side is a polynomial in $u$ as
  manifest in \thmref{thm:symmetry}, though it is {\it a priori\/} a
  formal power series in $u$. Then \eqref{eq:ZZ} is an equality in a
  formal power series in $u^{-1}$. Otherwise it contains infinitely
  many both negative and positive powers of $u$, and hence is not
  well-defined.

  (2) The class $s_\lambda[\calE/\{ y=0\}]$ is {\it not\/} integral
  (even for $\lambda = \square$). But the denominator is somewhat
  controllable as follows. Let us rewrite the above as
  \begin{equation*}
      Z(x,y;q,t,u,v)
  = \sum_{\lambda,\mu}
  \begin{aligned}[t]
  & v^{-|\lambda||\mu|}(tu)^{-\nicefrac{(|\lambda|+|\mu|)}2}
   s_\lambda[\frac{X}{(1-t)(1-q)}] s_\mu[\frac{Y}{(1-t)(1-q)}]
\\
  &\quad\times  
  \frac{
    \sum_{n=0}^\infty u^n \ch\mathbf R\Gamma_{X^{[n]}}(
    s_\lambda[\calE/[0]]\otimes s_\mu[\calE/[0]])
  }
  {\sum_{n=0}^\infty u^n \ch \mathbf R\Gamma_{X^{[n]}}(\shfO_{X^{[n]}})}.
  \end{aligned}
  \end{equation*}
  By \corref{cor:empty}, this is equal to
  \begin{equation*}
    \sum_{\lambda,\mu,\nu}v^{-|\lambda||\mu|}
    (tu)^{-\nicefrac{(|\lambda|+|\mu|)}2} N_{\lambda\mu}^\nu
    s_\lambda[\frac{X}{(1-t)(1-q)}] s_\mu[\frac{Y}{(1-t)(1-q)}]
    (\nabla s_\nu)[1-u], 
  \end{equation*}
  where $N_{\lambda\mu}^\sigma$ is the (ordinary)
  Littlewood–Richardson coefficient. We write
  \begin{equation*}
    s_\lambda(x)
    = \sum_\nu X_{\lambda\nu}(t) s_\nu[(1-t){X}]
    = \sum_\nu X_{\lambda\nu}(t) S_\nu(x;t),
  \end{equation*}
  where $S_\nu(x;t)$ is as in \cite[III\S4]{Mac}. Replacing $X$ by
  $X/(1-t)(1-q)$ we get
  \begin{equation*}
    s_\lambda[\frac{X}{(1-t)(1-q)}]
    = \sum_\nu X_{\lambda\nu}(t) s_\nu[\frac{X}{1-q}].
  \end{equation*}
  This means that though $Z$ is not an integral class in
  $K_{\C^*\times\C^*}(L_x^{[m]}\times L_y^{[n]})$, the denominator
  only comes from $X_{\lambda\nu}(t)$ and $X_{\mu\sigma}(q)$.
  The transition matrix $X_{\lambda\nu}(t)$ is given in
  \cite[III\S6]{Mac}, and the denominator is of the form
  \begin{equation*}
    b_\lambda(t)^{-1} = \prod_{i=1}^{l(\lambda)} \varphi_{m_i}(t)^{-1},
  \end{equation*}
  where $\lambda = (1^{m_1}2^{m_2}\cdots)$ and $\varphi_m(t) =
  (1-t)(1-t^2)\cdots (1-t^m)$.
\end{Remarks}

\begin{NB}
  Suppose that $F$ is a class supported on the subvariety
  $\pi^{-1}(S^n(\{ x=0\}))$. We claim that 
  \begin{equation*}
    s_\lambda[F],
  \end{equation*}
  is well-defined as a class supported on the same subvariety.
  By an analog of \corref{cor:K}, $F\in K_{\C^*\otimes \C^*}(X^{[n]})$
  is supported on the subvariety if and only if $\Phi(F)\in
  \Q(q,t)\otimes \Lambda$ satisfies $\Phi(F)[(1-q)X]\in \Z[q^{\pm 1},
  t^{\pm 1}]\otimes \Lambda$. So we may assume $\Phi(F) =
  s_\mu[X/(1-q)]$. Then
  \begin{equation*}
    \Phi(s_\lambda[F]) = s_\lambda[s_\mu[\frac{X}{1-q}]]
    = (s_\lambda[s_\mu])[\frac{X}{1-q}],
  \end{equation*}
  where $s_\lambda[s_\mu]$ is the plethystic substitution of $s_\mu$
  into $s_\lambda$.
  \begin{NB2}
    Consider the power sum case first. We have
    \begin{equation*}
      p_\lambda[p_\mu[\frac{X}{1-q}]]
      = p_\lambda[\prod_i \frac{x_1^{\mu_i} + x_2^{\mu_i} + \dots}{1-q^{\mu_i}}]
      = \prod_{i,j} \frac{x_1^{\mu_i\lambda_j} + x_2^{\mu_i\lambda_j} + \dots}
      {1-q^{\mu_i\lambda_j}}.
    \end{equation*}
    On the other hand, we have
    \begin{equation*}
      (p_\lambda[p_\mu])[\frac{X}{1-q}]
      = \prod_{i,j} p_{\mu_i\lambda_j}[\frac{X}{1-q}]
      = \prod_{i,j} \frac{p_{\mu_i\lambda_j}}{1-q^{\mu_i\lambda_j}}.
    \end{equation*}
    Thus these are equal in this case. We expand $s_\lambda = \sum_\nu
    X^\nu_\lambda p_\nu$. Then
    \begin{equation*}
      s_\lambda[s_\mu[\frac{X}{1-q}]]
      = \sum_{\nu_1,\nu_2} X^{\nu_1}_\lambda X^{\nu_2}_\mu 
      p_{\nu_1}[p_{\nu_2}[\frac{X}{1-q}]]
      = \sum_{\nu_1,\nu_2} X^{\nu_1}_\lambda X^{\nu_2}_\mu 
      (p_{\nu_1}[p_{\nu_2}])[\frac{X}{1-q}]
      = (s_\lambda[s_\mu])[\frac{X}{1-q}].
    \end{equation*}
  \end{NB2}
  Now it is known that $s_\lambda[s_\mu]\in \Lambda$ \cite[I
  \S8]{Mac}. Therefore we have the assertion.

  \begin{NB2}
    It seems that $\Phi(s_\lambda[F]) = s_\lambda[\Phi(F)]$ is {\it
      not\/} true.
  \end{NB2}
\end{NB}

\begin{NB}
  Recall $\mathcal P^* = \mathcal
  L^{-1}\otimes\varepsilon\otimes\mathcal P$, where $\varepsilon$ is
  the sign representation of $S_n$.
\end{NB}

\begin{NB}
\subsection{Refined Littlewood-Richardson coefficients}

Let $d_{\lambda\mu}^\nu(q,t)$ be the rational function defined by
\begin{equation*}
  \tilde H_{\lambda}(x;q,t) \tilde H_{\mu}(x;q,t)
  = \sum_\nu d_{\lambda\mu}^\nu(q,t) \tilde H_\nu(x;q,t).
\end{equation*}
This is the $(q,t)$-analog of the Littlewood-Richardson coefficients.
Note that this is a finite sum. It is clear that $|\nu| =
|\lambda|+|\mu|$. And it is also known that $\lambda$, $\mu$ must be
contained in $\nu$ \cite[???]{Mac}.

Then the right hand side of \eqref{eq:HilbHilb} is
\begin{equation*}
  \sum_\nu d_{\lambda\mu}^\nu(q^{-1},t^{-1}) \sum_{n=0}^\infty u^n
  \mathbf R\Gamma_{X^{[n]}}((\omega\nabla^{-1}
  \tilde H_\nu)[\uE;q,t])).
\end{equation*}
Using \eqref{eq:qtdim} we further rewrite this as
\begin{equation*}
  \sum_\nu d_{\lambda\mu}^\nu(q^{-1},t^{-1}) \Omega[\frac{u}{(1-q)(1-t)}]
  \tilde H_\nu[1-u;q,t]
  =   \Omega[\frac{u}{(1-q)(1-t)}] \sum_\nu 
  \frac{d_{\lambda\mu}^\nu(q^{-1},t^{-1})}{\Omega[u B_\nu(q,t)]}.
\end{equation*}
Therefore
\begin{equation*}
  \frac{\tilde H_\mu[1 - u A_\lambda(q,t);q,t]}
  {\Omega[u B_\lambda(q,t)]}
  = \sum_\nu 
  \frac{d_{\lambda\mu}^\nu(q^{-1},t^{-1})}{\Omega[u B_\nu(q,t)]}.
\end{equation*}
\end{NB}

\appendix
\section{Partitions}\label{sec:part}

We collect some notation on partitions.

A partition $\lambda$ is a non-decreasing sequence $\lambda =
(\lambda_1\ge\lambda_2\ge\lambda_3\ge\cdots)$ of nonnegative integers
such that $\lambda_N = 0$ for $N\gg 0$. We usually do not write the
entry $\lambda_i$ if it is zero. We identify $\lambda$ with a Young
diagram $Y_\lambda$ as in \cite{MR2095899}. The transpose of $\lambda$
is denoted by $\lambda^t$. We consider $Y_\lambda$ as a subset of
boxes $s$ in the $xy$-plane. For a box $s$ in the $xy$-plane and a
Young diagram $Y$, we denote by $a_Y(s)$, $l_Y(s)$, $a'(s)$, $l'(s)$
the arm-length, leg-length, arm-colength, leg-colength respectively.
The hook length is defined by $h_Y(s) = a_Y(s)+l_Y(s)+1$.  We also
write it $a_\lambda(s)$, $l_\lambda(s)$, $h_\lambda(s)$ if $Y$
corresponds to a partition $\lambda$. (We shall use only deal with
rank $1$ case, i.e., the case of Hilbert schemes of points.) When a
box $s$ is in the Young diagram, we denote by $s\in Y_\lambda$, or
simply by $s\in\lambda$.
Since the diagram is clear from the context (because we only consider
`rank $1$' case), we simply denote them by $a(s)$, $l(s)$, $h(s)$.

Unfortunately our convention of arm/leg-(co)length functions is
different from \cite{IKV,GIKV}. Our Young diagram is rotated by
$90^\circ$ anticlockwise.

We give some standard notations and useful formulas:
\begin{equation*}
  \begin{split}
   & |\lambda| = \sum_i \lambda_i
   = \# \{ s\in\lambda\} = \text{the number of boxes in $Y_\lambda$},
\\
   & n(\lambda) = \sum_i (i-1)\lambda_i
   = \sum_{s\in Y_\lambda} l'(s) = \sum_{s\in Y_\lambda} l_Y(s),
\\
   & n(\lambda^t) = \sum_{s\in Y_\lambda} a'(s) = \sum_{s\in Y_\lambda} a_Y(s),
\\
   & \kappa(\lambda) = 2(n(\lambda^t) - n(\lambda)),
\\
   & \|\lambda\|^2 = \sum_i \lambda_i^2
   = 2 n(\lambda^t) + |\lambda|,
\\
   & B_\lambda(q,t) 
   = \sum_{s\in\lambda} q^{a'(s)} t^{l'(s)} = B_{\lambda^t}(t,q),
\\
   & A_\lambda(q,t) = 1 - (1-t)(1-q) B_\lambda(q,t) = A_{\lambda^t}(t,q),
\\
   & c_\lambda(q,t) = \prod_{s\in\lambda} (1 - q^{a(s)} t^{l(s)+1}),
\\
   & c'_\lambda(q,t) = \prod_{s\in\lambda} (1 - q^{a(s)+1} t^{l(s)})
   = c_{\lambda^t}(t,q).
  \end{split}
\end{equation*}

\section{Symmetric polynomials}\label{sec:symm}

Our notation for symmetric polynomials is standard, and follows
\cite{Mac}. We denote by $e_k$, $h_k$, $p_k$ the $k^{\mathrm{th}}$
elementary, complete, and power sum symmetric functions. Schur
functions corresponding to a partition $\lambda$ is denoted by
$s_\lambda$.
\begin{NB}
\begin{equation*}
    \sum_{n=0}^\infty h_n t^n = \exp\left(
    \sum_{k=1}^\infty \frac{p_k t^k}k \right), \qquad
    \sum_{n=0}^\infty e_n t^n = \exp\left(-
    \sum_{k=1}^\infty \frac{p_k (-t)^k}k \right).
\end{equation*}
\begin{equation*}
   s_{\lambda}(q_2^{-\rho}) = s_\lambda(q^{\frac12},q^{\frac32},
   q^{\frac52},\dots)
   = q^{n(\lambda)+\frac{|\lambda|}2} \prod_{s\in\lambda}
    (1 - q^{h_\lambda(s)})^{-1}.
\end{equation*}
\end{NB}

The ring of symmetric functions $\varprojlim \Z[x_1,\dots, x_N]^{S_N}$
is denoted by $\Lambda$. The Schur functions $s_\lambda$ define an
integral base of $\Lambda$.

An inner product on $\Lambda$ is defined by
\begin{equation*}
  \langle s_\lambda, s_\mu\rangle = \delta_{\lambda\mu}.
\end{equation*}

The ring $\Lambda$ is isomorphic to the direct sum $\bigoplus_n
R(S_n)$ of Grothendieck group of representations of the symmetric
group $S_n$ of $n$ letters. Here the multiplication is given by
\begin{equation*}
  V \cdot W = \operatorname{Ind}_{S_m\times S_n}^{S_{m+n}}(V\boxtimes W).
\end{equation*}
See \cite[I\S7]{Mac}. A Schur function $s_\lambda$ corresponds to an
irreducible character. The isomorphism is called the {\it Frobenius
  characteristic map}.

The degree $n$ subspace of $\Lambda$ is denoted by $\Lambda^n$. It is
isomorphic to $R(S_n)$.

\bibliographystyle{myamsplain}
\bibliography{nakajima,mybib,refined}

\end{document}